\documentclass{article}
\usepackage{graphicx}
\usepackage{comment}
\usepackage{amsmath}
\usepackage{amsthm}
\usepackage{amssymb}
\usepackage[colorlinks=true,linkcolor=blue,citecolor=green]{hyperref}
\usepackage[margin=2.5cm]{geometry}
\usepackage{stmaryrd}
\usepackage{wasysym}
\usepackage{bussproofs}
\usepackage{mathtools}
\usepackage[all]{xy}
\usepackage{bussproofs}
\usepackage{mathrsfs}
\usepackage{forest}
\forestset{smullyan tableaux/.style={for tree={math content},where n children=1{!1.before computing xy={l=\baselineskip},!1.no edge}{},closed/.style={label=below:$\times$},},}
\newtheorem{theorem}{Theorem}[section]
\newtheorem{proposition}{Proposition}[section]

\newtheorem{lemma}{Lemma}[section]

\newtheorem{sublemma}{Lemma}[lemma]
\newtheorem{sublemmacase}{Case}[sublemma]
\theoremstyle{definition}
\newtheorem{definition}{Definition}[section]
\theoremstyle{remark}

\usepackage{longtable}
\usepackage{fancyhdr}
\usepackage{authblk}
\usepackage{soul}
\begin{document}
\providecommand{\keywords}[1]
{
  \small	
  \textbf{\textit{Keywords: }} #1
}
\title{Non-distributive relatives of ETL and NFL\thanks{The author would like to express his gratitude to his doctoral advisor professor Dmitryi Zaitsev for his fruitful advice and help as well as to professors Heinrich Wansing and Omori Hitoshi for their discussion of the earlier version of these results and kind permission on behalf of Prof.\ Wansing to present the earlier version on a~session of the research colloquium at Ruhr-Universit\"{a}t Bochum. The author also wishes to thank two anonymous referees whose detailed comments helped to enhance the quality of the paper.\\The work was partially supported by RUB Research School scholarship ‘Ph.D.-Exchange’.\\This is a~postprint version of the following paper --- \href{https://doi.org/10.1007/s11225-020-09904-3}{doi: 10.1007/s11225-020-09904-3}.}}
\author{Daniil Kozhemiachenko}
\affil{Department of Logic, Faculty of Philosophy, Lomonosov Moscow State University}
\date{}
\maketitle
\begin{abstract}
In this paper, we devise non-distributive relatives of Exactly True Logic (ETL) by Pietz and Riveccio and its dual (NFL) Non-Falsity Logic by Shramko, Zaitsev and Belikov. We consider two pre-orders which are algebraic counterparts of the ETL’s and NFL’s entailment relations on the De Morgan lattice \textbf{4}. We generalise these pre-orders and determine which distributive properties that hold on \textbf{4} are not forced by either of the pre-orders. We then construct relatives of ETL and NFL but lack such distributive properties. For these logics, we also devise a truth table semantics which uses non-distributive lattice \textbf{M3} as their lattice of truth values. We also provide analytic tableaux systems which work with sequents of the form $\phi\vdash\chi$. We also prove correctness and completeness results for these proof systems and provide a neat generalisation for non-distributive ETL- and NFL-like logics built over a certain family of non-distributive modular lattices.

\end{abstract}
\section{Introduction\label{introduction}}
\subsection{FDE and its relatives}\label{FDE}
J.M.\ Dunn and N.D.\ Belnap proposed their ‘useful four-valued logic’ (or FDE --- first-degree entailment) back in the 1970s in~\cite{Dunn1976,Belnap1977computer,Belnap1977fourvalued}.
	
Dunn formulated FDE in~\cite{Dunn2000} as a bi-consequence system $\mathbf{R}_{\mathrm{fde}}$ which uses sequents of the form $\phi\vdash_{\mathrm{FDE}}\chi$ (these are binary consequences) where $\phi$ and $\chi$ are propositional formulas over $\{\neg,\wedge,\vee\}$. $\mathbf{R}_{\mathrm{fde}}$ is the reflexive transitive closure of the following principles.
\[\begin{array}{cc}
\phi\wedge\chi\vdash_{\mathrm{FDE}}\phi&\phi\vdash_{\mathrm{FDE}}\phi\vee\chi\\
\phi\wedge\chi\vdash_{\mathrm{FDE}}\chi&\chi\vdash_{\mathrm{FDE}}\phi\vee\chi\\
\neg(\phi\wedge\chi)\dashv\vdash_{\mathrm{FDE}}\neg\phi\vee\neg\chi&\neg(\phi\vee\chi)\dashv\vdash_{\mathrm{FDE}}\neg\phi\wedge\neg\chi\\
\end{array}\]
\[\begin{array}{c}
\neg\neg\phi\dashv\vdash_{\mathrm{FDE}}\phi\\
\phi\wedge(\chi\vee\psi)\vdash_{\mathrm{FDE}}(\phi\wedge\chi)\vee(\phi\wedge\psi)
\end{array}\]
\[\begin{array}{cc}
\wedge_i\dfrac{\phi\vdash_{\mathrm{FDE}}\chi\quad\phi\vdash_{\mathrm{FDE}}\psi}{\phi\vdash_{\mathrm{FDE}}\chi\wedge\psi}&\vee_e\dfrac{\phi\vdash_{\mathrm{FDE}}\psi\quad\chi\vdash_{\mathrm{FDE}}\psi}{\phi\vee\chi\vdash_{\mathrm{FDE}}\psi}
\end{array}\]

First degree entailment is known to be complete w.r.t.\ the following semantics built upon four values: \textbf{T}~(true and not false), \textbf{B} (both true and false), \textbf{N} (neither true, nor false), and \textbf{F} (false and not true)
\begin{center}
\begin{tabular}{ccc}
\begin{tabular}{c|c}
$\neg$&\\\hline
$\mathbf{T}$&$\mathbf{F}$\\
$\mathbf{B}$&$\mathbf{B}$\\
$\mathbf{N}$&$\mathbf{N}$\\
$\mathbf{F}$&$\mathbf{T}$\\
\end{tabular}
&
\begin{tabular}{c|cccc}
$\wedge$&$\mathbf{T}$&$\mathbf{B}$&$\mathbf{N}$&$\mathbf{F}$\\\hline
$\mathbf{T}$&$\mathbf{T}$&$\mathbf{B}$&$\mathbf{N}$&$\mathbf{F}$\\
$\mathbf{B}$&$\mathbf{B}$&$\mathbf{B}$&$\mathbf{F}$&$\mathbf{F}$\\
$\mathbf{N}$&$\mathbf{N}$&$\mathbf{F}$&$\mathbf{N}$&$\mathbf{F}$\\
$\mathbf{F}$&$\mathbf{F}$&$\mathbf{F}$&$\mathbf{F}$&$\mathbf{F}$\\
\end{tabular}
&
\begin{tabular}{c|cccc}
$\vee$&$\mathbf{T}$&$\mathbf{B}$&$\mathbf{N}$&$\mathbf{F}$\\\hline
$\mathbf{T}$&$\mathbf{T}$&$\mathbf{T}$&$\mathbf{T}$&$\mathbf{T}$\\
$\mathbf{B}$&$\mathbf{T}$&$\mathbf{B}$&$\mathbf{T}$&$\mathbf{B}$\\
$\mathbf{N}$&$\mathbf{T}$&$\mathbf{T}$&$\mathbf{N}$&$\mathbf{N}$\\
$\mathbf{F}$&$\mathbf{T}$&$\mathbf{B}$&$\mathbf{N}$&$\mathbf{F}$\\
\end{tabular}
\end{tabular}
\end{center}
where entailment relation can be defined as ‘at least truth’ preservation
$$\phi\vDash_{\mathrm{FDE}}\chi\leftrightharpoons\forall v:v(\phi)\in\{\mathbf{T},\mathbf{B}\}\Rightarrow v(\chi)\in\{\mathbf{T},\mathbf{B}\}$$
	
Recently, a new approach of providing relatives to FDE was proposed first by Pietz and Riveccio in~\cite{PietzRiveccio2013} and then by Shramko, Zaitsev and Belikov in~\cite{ShramkoZaitsevBelikov2017,ShramkoZaitsevBelikov2018}. The point was to change the set of the designated values but retain the definitions of connectives. In this fashion ‘Exactly True Logic’ (ETL) by Pietz and Riveccio~\cite{PietzRiveccio2013} is the logic whose entailment relation is ‘truth and non-falsity’ preservation
$$\phi\vDash_{\mathrm{ETL}}\chi\leftrightharpoons\forall v:v(\phi)=\mathbf{T}\Rightarrow v(\chi)=\mathbf{T}$$
and ‘Non-Falsity Logic’ (NFL) by Shramko et al.~\cite{ShramkoZaitsevBelikov2017,ShramkoZaitsevBelikov2018} is the logic whose entailment relation is ‘non-falsity’ preservation
$$\phi\vDash_{\mathrm{NFL}}\chi\leftrightharpoons\forall v:v(\phi)\neq\mathbf{F}\Rightarrow v(\chi)\neq\mathbf{F}$$
	
Roughly speaking, ETL can be acquired if we replace the disjunction elimination rule --- $\vee_e$ --- of $\mathbf{R}_{\mathrm{fde}}$ with the disjunctive syllogism: $$\mathbf{DS}:\neg\phi\wedge(\phi\vee\chi)\vdash\chi$$ The resulting system is formulated in a binary consequence\footnote{For alternative proof system for ETL cf.~\cite{WinteinMuskens2016}.} form in~\cite{ShramkoZaitsevBelikov2018} as follows
\[\begin{array}{cc}
\phi\wedge\chi\vdash_{\mathrm{ETL}}\phi&\phi\vee(\chi\vee\psi)\vdash_{\mathrm{ETL}}(\phi\vee\chi)\vee\psi\\
\phi\wedge\chi\vdash_{\mathrm{ETL}}\chi&\phi\vee(\chi\wedge\psi)\dashv\vdash_{\mathrm{ETL}}(\phi\vee\chi)\wedge(\phi\vee\psi)\\
\phi\vdash_{\mathrm{ETL}}\phi\vee\chi&\phi\vee\chi\dashv\vdash_{\mathrm{ETL}}\neg\neg\phi\vee\chi\\
\phi\vee\chi\vdash_{\mathrm{ETL}}\chi\vee\phi&\neg(\phi\wedge\chi)\vee\psi\dashv\vdash_{\mathrm{ETL}}(\neg\phi\vee\neg\chi)\vee\psi\\
\phi\vee\phi\vdash_{\mathrm{ETL}}\phi&\neg(\phi\vee\chi)\vee\psi\dashv\vdash_{\mathrm{ETL}}(\neg\phi\wedge\neg\chi)\vee\psi\\
&\phi\wedge(\neg\phi\vee\chi)\vdash_{\mathrm{ETL}}\chi
\end{array}\]
\[\begin{array}{cc}
\mathsf{cut}\dfrac{\phi\vdash_{\mathrm{ETL}}\chi\quad\chi\vdash_{\mathrm{ETL}}\psi}{\phi\vdash_{\mathrm{ETL}}\psi}&\wedge_i\dfrac{\phi\vdash_{\mathrm{ETL}}\chi\quad\phi\vdash_{\mathrm{ETL}}\psi}{\phi\vdash_{\mathrm{ETL}}\chi\wedge\psi}
\end{array}\]

In the same fashion, we can speak of NFL as the result of replacing the conjunction introduction rule --- $\wedge_i$ --- of $\mathbf{R}_{\mathrm{fde}}$ with the dual disjunctive syllogism: $$\mathbf{DDS}:\phi\vdash\neg\chi\vee(\chi\wedge\phi)$$ The resulting binary consequence system is as follows~\cite{ShramkoZaitsevBelikov2018}:
\[\begin{array}{cc}
\phi\vdash_{\mathrm{NFL}}\phi\vee\chi&(\phi\wedge\chi)\wedge\psi\vdash_{\mathrm{NFL}}\phi\wedge(\chi\wedge\psi)\\
\chi\vdash_{\mathrm{NFL}}\phi\vee\chi&\phi\wedge(\chi\vee\psi)\dashv\vdash_{\mathrm{NFL}}(\phi\wedge\chi)\vee(\phi\wedge\psi)\\
\phi\wedge\chi\vdash_{\mathrm{NFL}}\phi&\phi\wedge\chi\dashv\vdash_{\mathrm{NFL}}\neg\neg\phi\wedge\chi\\
\phi\wedge\chi\vdash_{\mathrm{NFL}}\chi\wedge\phi&\neg(\phi\wedge\chi)\wedge\psi\dashv\vdash_{\mathrm{NFL}}(\neg\phi\vee\neg\chi)\wedge\psi\\
\phi\vdash_{\mathrm{NFL}}\phi\wedge\phi&\neg(\phi\vee\chi)\wedge\psi\dashv\vdash_{\mathrm{NFL}}(\neg\phi\wedge\neg\chi)\wedge\psi\\
&\phi\vdash_{\mathrm{NFL}}\neg\chi\vee(\chi\wedge\phi)\\
\end{array}\]
\[\begin{array}{cc}
\mathsf{cut}\dfrac{\phi\vdash_{\mathrm{NFL}}\chi\quad\chi\vdash_{\mathrm{NFL}}\psi}{\phi\vdash_{\mathrm{NFL}}\psi}&\vee_e\dfrac{\phi\vdash_{\mathrm{NFL}}\psi\quad\chi\vdash_{\mathrm{NFL}}\psi}{\phi\vee\chi\vdash_{\mathrm{NFL}}\psi}
\end{array}\]
\subsection{Non-distributive logics}\label{LN}
The second source of our motivation is threefold. First, it is a non-dis\-tri\-bu\-ti\-ve logic $LN$ proposed in~\cite{BimboDunn2001}. $LN$~is a~cut-free sequent calculus which constitutes the positive fragment of FDE that proves neither $$\phi\wedge(\chi\vee\psi)\vdash(\phi\wedge\chi)\vee(\phi\wedge\psi)\text{ and }(\phi\vee\chi)\wedge(\phi\vee\psi)\vdash\phi\vee(\chi\wedge\psi)$$ nor $$\phi\wedge(\chi\vee\psi)\vdash(\phi\wedge\chi)\vee\psi$$
As pointed out in~\cite{BimboDunn2001}, $LN$ is a~natural calculus since it can be obtained by simply removing structural rules but retaining the usual introduction and elimination rules for conjunction and disjunction.
	
Second, observe that ETL and NFL can be construed as versions of Priest's logic of paradox from~\cite{Priest1979} and Kleene's strong three-valued logics with restriction on conjunction and disjunction in the following sense. ETL lacks disjunction elimination $$\dfrac{\phi\vdash\psi\quad\chi\vdash\psi}{\phi\vee\chi\vdash\psi}$$ while NFL lacks conjunction introduction $$\dfrac{\phi\vdash\chi\quad\phi\vdash\psi}{\phi\vdash\chi\wedge\psi}$$
The lack of distributive properties listed above usually places additional restrictions on disjunction and conjunction, so it is natural to ask what happens if we apply these restrictions.
	
Third, non-distributive lattices appear in the study of quantum logic (cf., e.g.~\cite{HandbookofQuantumLogicandQuantumStructures2009} for reference). These are orthomodular. Observe, however, that \textbf{4} equipped with the De Morgan negation is distributive and hence modular but is \textit{not} orthomodular.
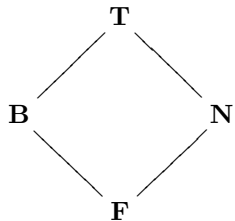
\[\xymatrix{
&\mathbf{T}&\\
\mathbf{B}\ar@{-}[ur]&&\mathbf{N}\ar@{-}[ul]\\
&\mathbf{F}\ar@{-}[ur]\ar@{-}[ul]&
}\]
It thus would be interesting to consider modular non-distributive lattices and ETL and NFL logics on them.
\subsection{Entailment as lattice pre-order}\label{preorders}
Recall that truth values for first-degree entailment, Exactly True logic and Non-Falsity Logic form a De Morgan lattice \textbf{4}.
	
It is also straightforward to see that three different entailment relations: $\vDash_{\mathrm{FDE}}$, $\vDash_{\mathrm{ETL}}$ and $\vDash_{\mathrm{NFL}}$ constitute three different pre-orders in \textbf{4}\footnote{Note, however, that $\vDash_{\mathrm{FDE}}$ can be defined as an order on \textbf{4}, as shown by Font in~\cite{Font1997}.}. Moreover, the following distributive properties hold for entailment relations of FDE, ETL and NFL.
\begin{itemize}
\item $\phi\wedge(\chi\vee\psi)\vDash(\phi\wedge\chi)\vee(\phi\wedge\psi)$.
\item $(\phi\wedge\chi)\vee(\phi\wedge\psi)\vDash\phi\wedge(\chi\vee\psi)$.
\item $\phi\vee(\chi\wedge\psi)\vDash(\phi\vee\chi)\wedge(\phi\vee\psi)$.
\item $(\phi\vee\chi)\wedge(\phi\vee\psi)\vDash\phi\vee(\chi\wedge\psi)$.
\item $\phi\wedge(\chi\vee\psi)\vDash(\phi\wedge\chi)\vee\psi$.
\item $(\phi\vee\chi)\wedge\psi\vDash\phi\vee(\chi\wedge\psi)$.
\end{itemize}
	
A reasonable question may be asked what happens to these when we generalise $\vDash_{\mathrm{ETL}}$ and $\vDash_{\mathrm{NFL}}$ to all bounded lattices.
\subsection{Plan of the paper}\label{plan}
The text is organised as follows. In \S\ref{distributivitypreorder}, we define properties of ETL- and NFL-like logics. We then introduce generalised versions of the $\vDash_{\mathrm{ETL}}$ and $\vDash_{\mathrm{NFL}}$ entailment relations and investigate which distributive properties they force. Then in~\S\ref{truthtables} we introduce a truth table semantics for ETL and NFL relatives $\mathrm{ETL}_\mathbf{M3}$ and $\mathrm{NFL}_{\mathbf{M3}}$ based upon \textbf{M3} lattice, and also present a family of $\mathbf{Mn}$ lattices who differ from \textbf{M3} in their number of midpoints. In~\S\ref{nDETLproofsystem} we construct a unified analytic tableaux calculus for $\mathrm{ETL}_\mathbf{M3}$ and $\mathrm{NFL}_{\mathbf{M3}}$, and also provide motivation as well as correctness and completeness proofs (Theorems~\ref{M3correctness}--\ref{M3NFLcompleteness}) for these systems. In~\S\ref{generalisation} we generalise our result to a~family of ETL- and NFL-like logics built upon $\mathbf{Mn}$ lattices (Theorems~\ref{ETLMncompleteness}--\ref{NFLMncompleteness}); we also define $\mathrm{ETL}_{\mathbf{M}\omega}$ and $\mathrm{NFL}_{\mathbf{M}\omega}$ --- the ETL- and NFL-like logics over $\mathbf{M}\omega$ ($\mathbf{Mn}$ lattice with $\omega$ midpoints), and provide tableaux for them; we then show that $\mathrm{ETL}_{\mathbf{M}\omega}$ and $\mathrm{NFL}_{\mathbf{M}\omega}$ are ETL- and NFL-like logics of all bounded lattices (Theorems~\ref{ETLomega} and~\ref{NFLomega}). Finally, we wrap up our work and set goals for future research in~\S\ref{conclusion}.
\section[Distributive properties and pre-orders]{Distributive properties and pre-orders on bounded lattices}\label{distributivitypreorder}
For any bounded lattice $\mathfrak{L}=\langle L,\wedge,\vee,\neg,\top,\bot\rangle$ we have two matrices --- $\mathrm{ETL}_{\mathfrak{L}}=\langle\mathfrak{L},\{\top\}\rangle$ and $\mathrm{NFL}_{\mathfrak{L}}=\langle\mathfrak{L},L\setminus\{\bot\}\rangle$ --- that generalise entailment relations $\vDash_{\mathrm{ETL}}$ and $\vDash_{\mathrm{NFL}}$. Note that the entailment relations on $\mathrm{ETL}_{\mathfrak{L}}$ and $\mathrm{NFL}_{\mathfrak{L}}$ which we will further designate $\vDash_{\mathrm{ETL}_{\mathfrak{L}}}$ and $\vDash_{\mathrm{NFL}_{\mathfrak{L}}}$, respectively, constitute pre-orders on $\mathfrak{L}$.

Note, however, that the lattice $\mathbf{4}$ has some additional properties: first and foremost, De Morgan as well as double negation laws hold while negation is not lattice complement. With this in mind, let us now also generalise exactly true and Non-Falsity Logics.
\begin{definition}[ETL- and NFL-like logics]\label{ETLNFLlike}
For any bounded lattice $\mathfrak{L}$ and matrices $\mathrm{ETL}_{\mathfrak{L}}$ and $\mathrm{NFL}_{\mathfrak{L}}$ we will further call the logics of these matrices \textit{ETL-like logic on $\mathfrak{L}$} (\textit{NFL-like logic on $\mathfrak{L}$}, respectively) if the conditions listed below hold.
\begin{enumerate}
\item The logic enjoys the following version De Morgan laws ($v$ being a valuation in $\mathrm{ETL}_{\mathfrak{L}}$ or $\mathrm{NFL}_{\mathfrak{L}}$, respectively):
\[\forall v:v(\neg(\phi\wedge\chi))=v(\neg\phi\vee\neg\chi)\quad\forall v:v(\neg(\phi\vee\chi))=v(\neg\phi\wedge\neg\chi)\] and double negation law: \[\forall v:v(\neg\neg\phi)=v(\phi)\]
\item If $v(\phi\vee\neg\phi)=\top$, then $v(\phi)\in\{\top,\bot\}$.
\item $\phi\wedge\neg\phi\vDash_{\mathrm{ETL}_{\mathfrak{L}}}\chi$ and $\phi\vDash_{\mathrm{NFL}_{\mathfrak{L}}}\chi\vee\neg\chi$.
\item $\phi\wedge\neg\phi\nvDash_{\mathrm{NFL}_{\mathfrak{L}}}\chi$ and $\phi\nvDash_{\mathrm{ETL}_{\mathfrak{L}}}\chi\vee\neg\chi$.
\item $\neg\phi\wedge(\phi\vee\chi)\vDash_{\mathrm{ETL}_{\mathfrak{L}}}\chi$ and $\phi\vDash_{\mathrm{NFL}_{\mathfrak{L}}}\neg\chi\vee(\chi\wedge\phi)$.
\end{enumerate}
\end{definition}
Note also that $\vDash_{\mathrm{ETL}_{\mathfrak{L}}}$ and $\vDash_{\mathrm{NFL}_{\mathfrak{L}}}$ for the same lattice $\mathfrak{L}$ are completely dual, so we will further usually consider ETL-like logics in detail while mentioning NFL-like logics only briefly.

We now want to find out which distributive properties listed in section~\ref{preorders} are forced by either $\vDash_{\mathrm{ETL}_{\mathfrak{L}}}$ or $\vDash_{\mathrm{NFL}_{\mathfrak{L}}}$.
	
The following lemmas are easy to obtain.
\begin{lemma}\label{ETLdistributivity}
For any matrix $\mathrm{ETL}_{\mathfrak{L}}$ over a bounded lattice $\mathfrak{L}$ the following holds.
\begin{enumerate}
\item $\phi\wedge(\chi\vee\psi)\vDash_{\mathrm{ETL}_{\mathfrak{L}}}(\phi\wedge \chi)\vee(\phi\wedge\psi)$.
\item $(\phi\wedge \chi)\vee(\phi\wedge\psi)\vDash_{\mathrm{ETL}_{\mathfrak{L}}}\phi\wedge(\chi\vee\psi)$.
\item $\phi\vee(\chi\wedge\psi)\vDash_{\mathrm{ETL}_{\mathfrak{L}}}(\phi\vee \chi)\wedge(\phi\vee\psi)$.
\item $\phi\wedge(\chi\vee\psi)\vDash_{\mathrm{ETL}_{\mathfrak{L}}}(\phi\wedge \chi)\vee\psi$.
\item $(\phi\vee \chi)\wedge\psi\vDash_{\mathrm{ETL}_{\mathfrak{L}}}\phi\vee(\chi\wedge\psi)$.
\end{enumerate}
\end{lemma}
We will work here in the lattice upon which the matrix is built. It suffices to show that for all $a,b,c\in\mathfrak{L}$ and for any valuation $v$ if the left-hand side is equal to $\top$ under $v$, then so is the right-hand side. 
\begin{sublemma}
\end{sublemma}
\begin{proof}
Indeed, let
\[\begin{array}{rcl}
a\wedge(b\vee c)=\top&\Rightarrow&a=\top\&b\vee c=\top\\
&\Rightarrow&a\wedge b=b\&a\wedge c=c\\
&\Rightarrow&(a\wedge b)\vee(a\wedge c)=\top
\end{array}\]
\begin{sublemma}
\end{sublemma}
Now, let $a\wedge(b\vee c)\neq\top$. We have three cases.
\begin{sublemmacase}
$a\neq\top$ and $b\vee c\neq\top$
\end{sublemmacase}
In this case $b,c\neq\top$. It is straightforward to obtain that $(a\wedge b)\vee(a\wedge c)\neq\top$.
\begin{sublemmacase}
$a\neq\top$ and $b\vee c=\top$
\end{sublemmacase}
We have four options here. Either $b=c=\top$, or $b=\top$ but $c\neq\top$, or $b\neq\top$ but $c=\top$, or $b\neq\top$ and $c\neq\top$.
		
In the first case $(a\wedge b)\vee(a\wedge c)=a\neq\top$. In the second case we have several possibilities.
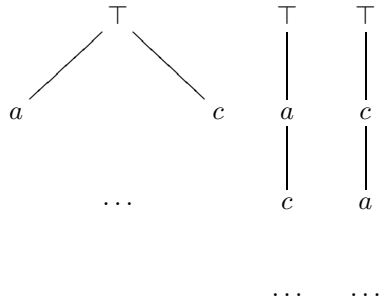
\begin{longtable}{ccc}
$$\xymatrix{
&\top&\\
a\ar@{-}[ur]&&c\ar@{-}[ul]\\
&\dots&
}$$
&
$$\xymatrix{
\top\\
a\ar@{-}[u]\\
c\ar@{-}[u]\\
\dots
}$$
&
$$\xymatrix{
\top\\
c\ar@{-}[u]\\
a\ar@{-}[u]\\
\dots
}$$
\end{longtable}
	
It is easy to see, that $(a\wedge b)\vee(a\wedge c)\neq\top$ in each case.
	
The third case can be tackled in the same fashion as the second one.
		
Finally, in the fourth case we have the following options (dashed lines denote that it is possible that $a\vee b\neq\top$ and $a\vee c\neq\top$).
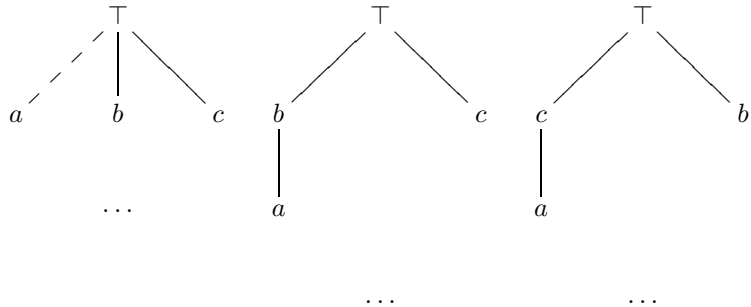
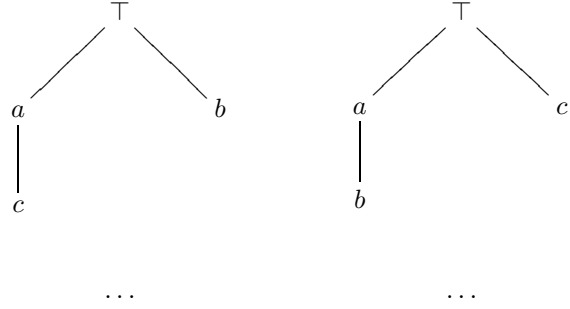
\begin{longtable}{ccc}
$$\xymatrix{
&\top&\\
a\ar@{--}[ur]&b\ar@{-}[u]&c\ar@{-}[ul]\\
&\dots&
}$$
&
$$\xymatrix{
&\top&\\
b\ar@{-}[ur]&&c\ar@{-}[ul]\\
a\ar@{-}[u]&&\\
&\dots&
}$$
&
$$\xymatrix{
&\top&\\
c\ar@{-}[ur]&&b\ar@{-}[ul]\\
a\ar@{-}[u]&&\\
&\dots&
}$$
\end{longtable}
\begin{longtable}{cc}
$$\xymatrix{
&\top&\\
a\ar@{-}[ur]&&b\ar@{-}[ul]\\
c\ar@{-}[u]&&\\
&\dots&&
}$$
&
$$\xymatrix{
&\top&\\
a\ar@{-}[ur]&&c\ar@{-}[ul]\\
b\ar@{-}[u]&&\\
&\dots&&
}$$
\end{longtable}
		
It is obvious, that $(a\wedge b)\vee(a\wedge c)\neq\top$ in each case.
\begin{sublemmacase}
$a=\top$ and $b\vee c\neq\top$
\end{sublemmacase}
In this case $(a\wedge b)\vee(a\wedge c)=b\vee c\neq\top$.
\end{proof}
\begin{sublemma}
\end{sublemma}
\begin{proof}
Let $(a\vee b)\wedge(a\vee c)\neq\top$.
\[\begin{array}{rcl}
(a\vee b)\wedge(a\vee c)\neq\top&\Rightarrow&a\vee b\neq\top\text{ or }a\vee c\neq\top\\
&\Rightarrow&a\neq\top\&b\neq\top\text{ or }a\neq\top\&c\neq\top
\end{array}\]
		
Consider the first case. We have the following options.
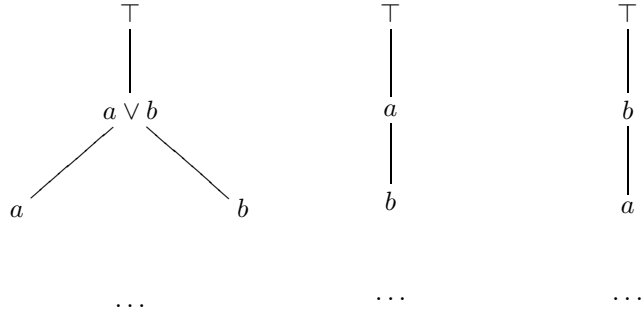
\begin{longtable}{ccc}
$$\xymatrix{
&\top&\\
&a\vee b\ar@{-}[u]&\\
a\ar@{-}[ur]&&b\ar@{-}[ul]\\
&\ldots&
}$$
&
$$\xymatrix{
&\top&\\
&a\ar@{-}[u]&\\
&b\ar@{-}[u]&\\
&\ldots&
}$$
&
$$\xymatrix{
&\top&\\
&b\ar@{-}[u]&\\
&a\ar@{-}[u]&\\
&\ldots&
}$$
\end{longtable}
It is evident that $a\vee(b\wedge c)\neq\top$.
		
The second case can be handled in the same fashion.
\end{proof}
\begin{sublemma}
\end{sublemma}
\begin{proof}
\[\begin{array}{rcl}
a\wedge(b\vee c)=\top&\Rightarrow&a=\top\&b\vee c=\top\\
&\Rightarrow&a\wedge b=b\\
&\Rightarrow&(a\wedge b)\vee c=b\vee c=\top
\end{array}\]
\end{proof}
\begin{sublemma}
\end{sublemma}
\begin{proof}
\[\begin{array}{rcl}
(a\vee b)\wedge c=\top&\Rightarrow&a\vee b=\top\&c=\top\\
&\Rightarrow&b\wedge c=b\\
&\Rightarrow&a\vee(b\wedge c)=a\vee b=\top
\end{array}\]
\end{proof}
\begin{lemma}\label{NFLdistributivity}
For any matrix $\mathrm{NFL}_{\mathfrak{L}}$ over a bounded lattice $\mathfrak{L}$ the following holds.
\begin{enumerate}
\item $\phi\vee(\chi\wedge\psi)\vDash_{\mathrm{NFL}_{\mathfrak{L}}}(\phi\vee\chi)\wedge(\phi\vee\psi)$.
\item $(\phi\vee\chi)\wedge(\phi\vee\psi)\vDash_{\mathrm{NFL}_{\mathfrak{L}}}\phi\vee(\chi\wedge\psi)$.
\item $(\phi\wedge\chi)\vee(\phi\wedge\psi)\vDash_{\mathrm{NFL}_{\mathfrak{L}}}\phi\wedge(\chi\vee\psi)$.
\item $\phi\wedge(\chi\vee\psi)\vDash_{\mathrm{NFL}_{\mathfrak{L}}}(\phi\wedge\chi)\vee\psi$.
\item $(\phi\vee\chi)\wedge\psi\vDash_{\mathrm{NFL}_{\mathfrak{L}}}\phi\vee(\chi\wedge\psi)$.
\end{enumerate}
\end{lemma}
\begin{proof}
By dualisation of Lemma~\ref{ETLdistributivity}.
\end{proof}
\section{Semantics for non-distributive relatives of ETL and NFL}\label{truthtables}
Now we need to find a bounded (preferably, finite) non-distributive lattice that possesses only needed distributive properties so that we can use it as a semantic framework (a reader interested in a thorough overview of results related to algebraic semantics of finite-valued logics may find them, e.g., in~\cite{FontJansana2009}). The most obvious candidates here are \textbf{M3} and \textbf{N5} lattices.
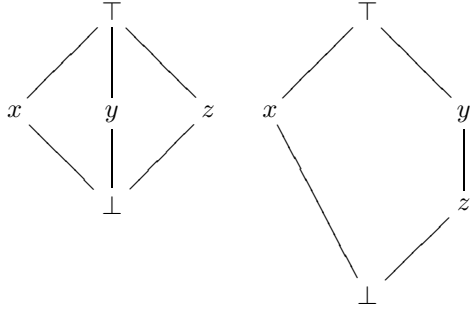
\begin{longtable}{cc}
$\xymatrix{
&\top&\\
x\ar@{-}[ur]&y\ar@{-}[u]&z\ar@{-}[ul]\\
&\bot\ar@{-}[ur]\ar@{-}[ul]\ar@{-}[u]&\\
}$
&
$\xymatrix{
&\top&\\
x\ar@{-}[ur]&&y\ar@{-}[ul]\\
&&z\ar@{-}[u]\\
&\bot\ar@{-}[ur]\ar@{-}[uul]&\\
}$
\end{longtable}
	
Observe also that the following version of the modular law holds for any $\mathrm{ETL}_{\mathfrak{L}}$ and $\mathrm{NFL}_{\mathfrak{L}}$ matrices on a bounded lattice $\mathfrak{L}$.
\[\phi\vDash_{\mathrm{ETL}_{\mathfrak{L}}}\chi\Rightarrow\phi\vee(\psi\wedge\chi)\vDash_{\mathrm{ETL}_{\mathfrak{L}}}(\phi\vee\psi)\wedge\chi\text{ and }(\phi\vee\psi)\wedge\chi\vDash_{\mathrm{ETL}_{\mathfrak{L}}}\phi\vee(\psi\wedge\chi)\]
\[\phi\vDash_{\mathrm{NFL}_{\mathfrak{L}}}\chi\Rightarrow\phi\vee(\psi\wedge\chi)\vDash_{\mathrm{NFL}_{\mathfrak{L}}}(\phi\vee\psi)\wedge\chi\text{ and }(\phi\vee\psi)\wedge\chi\vDash_{\mathrm{NFL}_{\mathfrak{L}}}\phi\vee(\psi\wedge\chi)\]
Note, however that $(\phi\vee\chi)\wedge(\phi\vee\psi)\vDash_{\mathrm{ETL}_{\mathbf{N5}}}\phi\vee(\chi\wedge\psi)$ and $\phi\wedge(\chi\vee\psi)\vDash_{\mathrm{NFL}_{\mathbf{N5}}}(\phi\wedge\chi)\vee(\phi\wedge\psi)$ hold. So, $\vDash_{\mathrm{ETL}_{\mathbf{N5}}}$ and $\vDash_{\mathrm{NFL}_{\mathbf{N5}}}$ \textit{are distributive} although, clearly, \textbf{N5} is \textit{non-distributive} w.r.t.\ usual ‘upward’ order. On the other hand, both $\vDash_{\mathrm{ETL}_{\mathbf{M3}}}$, $\vDash_{\mathrm{NFL}_{\mathbf{M3}}}$, and \textbf{M3} itself are not distributive and also modular.
	
Due to lemmas~\ref{ETLdistributivity} and~\ref{NFLdistributivity} we don't need to check matrices with more elements since any bounded lattice will lack at most one distributive property provable in ETL or NFL. Observe, however, that ETL (and NFL) logics built upon \textbf{M3} do not coincide with ETL and NFL over all bounded lattices. Indeed,
\begin{eqnarray}\label{ETLM3neqETLM4}
p\!\vee\!(((p\!\vee\!q)\!\vee\!(r\!\wedge\!s))\!\wedge\!(r\!\wedge(q\!\vee\!s)))\!\vDash_{\mathrm{ETL}_{\mathbf{M3}}}\!p\!\vee\!((q\!\vee\!(r\!\wedge\!s))\!\wedge\!(r\!\vee\!(p\!\wedge\!s)))
\end{eqnarray}
although this will not hold in the following lattice which we designate \textbf{M4}.
\[\xymatrix{
&&\top&&\\
{a}\ar@{-}[urr]&{b}\ar@{-}[ur]&&{c}\ar@{-}[ul]&{d}\ar@{-}[ull]\\
&&\bot\ar@{-}[ur]\ar@{-}[ul]\ar@{-}[urr]\ar@{-}[ull]&&\\
}\]
	
Lemmas~\ref{ETLdistributivity} and~\ref{NFLdistributivity} also imply that if we want to have a~non-distributive relatives of ETL and NFL based upon bounded lattices of truth values, we can drop only axioms $(\phi\vee\chi)\wedge(\phi\vee\psi)\vdash_{\mathrm{ETL}}\phi\vee(\chi\wedge\psi)$ and $\phi\wedge(\chi\vee\psi)\vdash_{\mathrm{NFL}}(\phi\vee\chi)\wedge(\phi\vee\psi)$.
	
In this section we will construct semantics for relatives of ETL and NFL lacking $$(\phi\vee\chi)\wedge(\phi\vee\psi)\vDash_{\mathrm{ETL}}\phi\vee(\chi\wedge\psi)$$ and $$\phi\wedge(\chi\vee\psi)\vDash_{\mathrm{NFL}}(\phi\vee\chi)\wedge(\phi\vee\psi)$$ which are based upon \textbf{M3}. We will then in \S\ref{generalisation} provide ETL- and NFL-like logics for a specific family of modular and non-distributive lattices w.r.t.\ $\preccurlyeq_{\mathrm{ETL}}$ and $\preccurlyeq_{\mathrm{NFL}}$.
	
We will further call these relatives ‘non-distributive ETL’ and ‘non-distributive NFL’ if we mean an arbitrary logic or the whole family. If we mean a logic for a~particular lattice, we will denote it $\mathrm{ETL}_{\mathfrak{L}}$ or $\mathrm{NFL}_{\mathfrak{L}}$ with $\mathfrak{L}$ being that lattice.
	
Again, just as in case of original Exactly True and Non-Falsity Logics, $\mathrm{ETL}_{\mathbf{M3}}$ or $\mathrm{NFL}_{\mathbf{M3}}$ use the same set of truth values and the same definitions of connectives. Their only difference lies in their entailment relations.
	
A valuation $v$ is a function mapping propositional variables to $\{\mathbf{T},\mathbf{B},\mathbf{0},\mathbf{N},\mathbf{F}\}$.
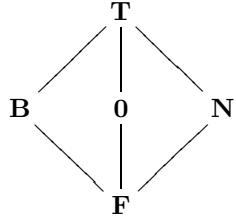
\[\xymatrix{
&\mathbf{T}&\\
\mathbf{B}\ar@{-}[ur]&\mathbf{0}\ar@{-}[u]&\mathbf{N}\ar@{-}[ul]\\
&\mathbf{F}\ar@{-}[ur]\ar@{-}[ul]\ar@{-}[u]&\\
}\]
	
The truth values of compound formulas are determined via the following truth tables.
\begin{center}
\begin{tabular}{ccc}
\begin{tabular}{c|c}
$\neg$&\\\hline
$\mathbf{T}$&$\mathbf{F}$\\
$\mathbf{B}$&$\mathbf{B}$\\
$\mathbf{0}$&$\mathbf{0}$\\
$\mathbf{N}$&$\mathbf{N}$\\
$\mathbf{F}$&$\mathbf{T}$\\
\end{tabular}
&
\begin{tabular}{c|ccccc}
$\wedge$&$\mathbf{T}$&$\mathbf{B}$&$\mathbf{0}$&$\mathbf{N}$&$\mathbf{F}$\\\hline
$\mathbf{T}$&$\mathbf{T}$&$\mathbf{B}$&$\mathbf{0}$&$\mathbf{N}$&$\mathbf{F}$\\
$\mathbf{B}$&$\mathbf{B}$&$\mathbf{B}$&$\mathbf{F}$&$\mathbf{F}$&$\mathbf{F}$\\
$\mathbf{0}$&$\mathbf{0}$&$\mathbf{F}$&$\mathbf{0}$&$\mathbf{F}$&$\mathbf{F}$\\
$\mathbf{N}$&$\mathbf{N}$&$\mathbf{F}$&$\mathbf{F}$&$\mathbf{N}$&$\mathbf{F}$\\
$\mathbf{F}$&$\mathbf{F}$&$\mathbf{F}$&$\mathbf{F}$&$\mathbf{F}$&$\mathbf{F}$\\
\end{tabular}
&
\begin{tabular}{c|ccccc}
$\vee$&$\mathbf{T}$&$\mathbf{B}$&$\mathbf{0}$&$\mathbf{N}$&$\mathbf{F}$\\\hline
$\mathbf{T}$&$\mathbf{T}$&$\mathbf{T}$&$\mathbf{T}$&$\mathbf{T}$&$\mathbf{T}$\\
$\mathbf{B}$&$\mathbf{T}$&$\mathbf{B}$&$\mathbf{T}$&$\mathbf{T}$&$\mathbf{B}$\\
$\mathbf{0}$&$\mathbf{T}$&$\mathbf{T}$&$\mathbf{0}$&$\mathbf{T}$&$\mathbf{0}$\\
$\mathbf{N}$&$\mathbf{T}$&$\mathbf{T}$&$\mathbf{T}$&$\mathbf{N}$&$\mathbf{N}$\\
$\mathbf{F}$&$\mathbf{T}$&$\mathbf{B}$&$\mathbf{0}$&$\mathbf{N}$&$\mathbf{F}$\\
\end{tabular}
\end{tabular}
\end{center}
	
Truth tables are quite expectable since $\wedge$ and $\vee$ coincide with meet and join of \textbf{M3} while negation allows for the following version of De Morgan laws.
\[\forall v:v(\neg(\phi\wedge\chi))=v(\neg\phi\vee\neg\chi)\quad\forall v:v(\neg(\phi\vee\chi))=v(\neg\phi\wedge\neg\chi)\]
	
Entailment relations are as follows.
\begin{definition}\label{entailments}
\[\phi\vDash_{\mathrm{ETL}_{\mathbf{M3}}}\chi\Leftrightarrow\forall v:v(\phi)=\mathbf{T}\Rightarrow v(\chi)=\mathbf{T}\]
\[\phi\vDash_{\mathrm{NFL}_{\mathbf{M3}}}\chi\Leftrightarrow\forall v:v(\phi)\neq\mathbf{F}\Rightarrow v(\chi)\neq\mathbf{F}\]
\end{definition}
	
As one can see, these entailment relations conform to the requirements from Definition~\ref{ETLNFLlike}. Moreover, it is tedious but straightforward to check that all axioms and rules of ETL and NFL except for $$(\phi\vee\chi)\wedge(\phi\vee\psi)\vdash_{\mathrm{ETL}}\phi\vee(\chi\wedge\psi)$$ and $$\phi\wedge(\chi\vee\psi)\vdash_{\mathrm{NFL}}(\phi\vee\chi)\wedge(\phi\vee\psi)$$ do hold.

In this paper we will be concerned with ETL- and NFL-like logics built upon a certain family of lattices which we call $\mathbf{Mn}$ lattices (Fig.~\ref{Mnlatticesfigure}). One can easily see that ETL- and NFL-like logics can be constructed if conjunction and disjunction coincide with the meet and join of these lattices and negation flips $\top$ and $\bot$ leaving elements designated with numbers intact.

The logics built upon the first two lattices are, obviously, \textbf{K3} and Priest's logic of paradox, and ETL and NFL respectively.
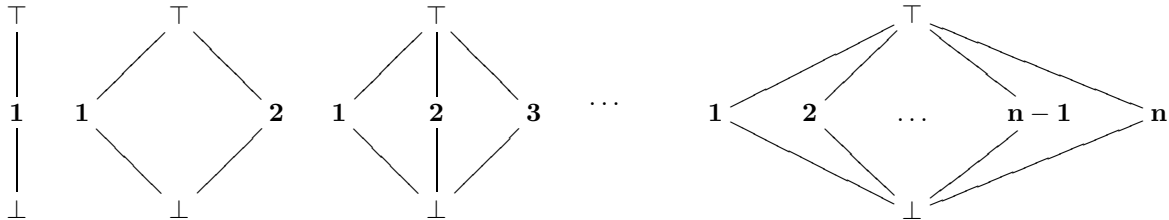
\begin{figure}[h!]
\begin{center}
\begin{tabular}{ccccc}
$\xymatrix{\top\\\ar@{-}[d]\ar@{-}[u]\mathbf{1}\\\bot}$
&
$\xymatrix{
&\top&\\
\mathbf{1}\ar@{-}[ur]&&\mathbf{2}\ar@{-}[ul]\\
&\bot\ar@{-}[ur]\ar@{-}[ul]&}$
&
$\xymatrix{
&\top&\\
\mathbf{1}\ar@{-}[ur]&\mathbf{2}\ar@{-}[u]&\mathbf{3}\ar@{-}[ul]\\
&\bot\ar@{-}[ur]\ar@{-}[ul]\ar@{-}[u]&\\}$
&
$\xymatrix{\\\ldots\\}$
&
\end{tabular}
\begin{tabular}{c}
$\xymatrix{
&&\top&&\\
\mathbf{1}\ar@{-}[urr]&\mathbf{2}\ar@{-}[ur]&\ldots&\mathbf{n-1}\ar@{-}[ul]&\mathbf{n}\ar@{-}[ull]\\
&&\bot\ar@{-}[ur]\ar@{-}[ul]\ar@{-}[urr]\ar@{-}[ull]&&\\}$
\end{tabular}
\end{center}
\caption{$\mathbf{Mn}$ lattices. From left to right: the three-element lattice, \textbf{4}, \textbf{M3}, etc. In general, $\mathbf{Mn}$ lattice contains $n$ elements on its middle level.}\label{Mnlatticesfigure}
\end{figure}

Our goal now is to provide proof systems formalising these semantics.
\section{Analytic tableaux}\label{nDETLproofsystem}
\subsection{Motivation}\label{motivation}
As we have seen, both ETL- and NFL-like logics for \textbf{M3} as well as any ETL- and NFL-like logics whose truth values comprise a finite bounded lattice are actually logics of finite matrices. As such, there is an algorithmic way to construct analytic calculi for them using ‘$n$-sided’ sequents (cf.~\cite{BaazFermullerSalzerZach1996} for the system description and~\cite{BaazFermullerSalzer2000} for theoretical work\footnote{A reader may also want to use the programme itself which can be downloaded at \url{https://www.logic.at/multlog/}.}). Another approach which is closer to tableaux would be to use ‘multiple-conclusion’ calculi (cf., e.g.~\cite{MarcelinoCaleiro2019} for recent results which show how to acquire multiple conclusion calculi from finite non-deterministic matrices).
	
Hence, there is a need to motivate a new proof system. Indeed, the above mentioned approaches easily produce analytic calculi whose completeness is also easy to prove. On the other hand, the number of ‘sides’ of sequents and rules for these calculi grows proportionally to the number of truth values of the logic\footnote{In particular, there would be five rules for each of $\{\neg,\wedge,\vee\}$ if we constructed such a calculus for $\mathrm{ETL}_{\mathbf{M3}}$ or $\mathrm{NFL}_{\mathbf{M3}}$.}. Moreover, proof systems for related logics --- like FDE, Kleene's strong three-valued logic and LP which usually differ only by one axiom (cf.~\cite{ShramkoZaitsevBelikov2017,ShramkoZaitsevBelikov2018} for details) --- will not be related in such a clear way if we use $n$-sided sequents since FDE is four-valued while LP and Kleene's logic are three-valued. Further, the multiple conclusion calculi would require the use of different separating formulas for different matrices, so we would lose the desired generalisability of our tableaux.
	
This is why, in the following section we will provide an analytic tableaux system for sequents of the form $\phi\vdash\chi$ which has several properties which we deem more desirable and more convenient than those of analytic calculi with $n$-sided sequents.
\begin{enumerate}
\item We use the same labels for formulas for any ETL- or NFL-like logic as long as its truth values comprise an \textbf{Mn} lattice (cf.~\S\ref{generalisation}).
\item The rules and definition of a closed tableau for any two ETL- and NFL-like logics built upon the same lattice will be the same. The only difference would be which tableaux we will build in order to prove or disprove a~sequent.
\item The tableaux can be simply adjusted without adding new rules or labels for formulas for any ETL- or NFL-like logic as long as its truth values comprise a lattice of a certain kind.
\end{enumerate}
\subsection[Unified tableaux]{A unified tableau calculus for $\mathrm{ETL}_{\mathbf{M3}}$ and $\mathrm{NFL}_{\mathbf{M3}}$}\label{tableaux}
We are now ready to provide our tableaux. They are going to be presented in a fashion similar to those of Smullyan's~\cite{Smullyan1995} in the sense that we will have two kinds of rules --- $\alpha$ (which do not require splitting the branch) and $\beta$ (which split the branch) ones --- which will extend the branch with new labelled formulas. Moreover, our tableaux are \textit{refutation} tableaux in the sense that in order to prove $\phi\vdash\chi$ we try to establish that there is no refuting valuation $v$ such that $v(\phi)=\mathbf{T}$ but $v(\chi)\neq\mathbf{T}$ if we are considering $\mathrm{ETL}_{\mathbf{M3}}$\footnote{For $\mathrm{NFL}_{\mathbf{M3}}$ we would try to show that there is no valuation $v$ such that $v(\phi)\neq\mathbf{F}$ but $v(\chi)=\mathbf{F}$.}.
	
As we have already mentioned, $\mathrm{ETL}_{\mathbf{M3}}$ and $\mathrm{NFL}_{\mathbf{M3}}$ use the same rules and the same criteria for the branch closure.
\begin{definition}[Tableaux for $\mathrm{ETL}_{\mathbf{M3}}$ and $\mathrm{NFL}_{\mathbf{M3}}$]\label{M3tableaux}
We define a tableau as a downward branching tree each node of which contains a set of labelled formulas and labelled pairs of formulas. Any formula $\phi$ can be labelled $\mathfrak{t}[\phi]$, $\mathfrak{m}[\phi]$ or $\mathfrak{f}[\phi]$, moreover, a pair of formulas can be labelled $\phi\sim\chi$ or $\phi\nsim\chi$.

Let us clarify the meanings of our labels. As expected, $\mathfrak{t}[\phi]$ and $\mathfrak{f}[\phi]$ denote that for the valuation $v$ $v(\phi)=\mathbf{T}$ and $v(\phi)=\mathbf{F}$ respectively while $\mathfrak{m}[\phi]$ tells us that $v(\phi)\in\{\mathbf{B},\mathbf{0},\mathbf{N}\}$ but does not provide an explicit valuation for $\phi$. In order to overcome this obstacle, we use $\sim$ and $\nsim$ which tell us that two formulas have the same or different values from $\{\mathbf{B},\mathbf{0},\mathbf{N}\}$.

The branch can be extended by an application of one of the following rules if it contains the appropriate premise(s) (below $i=1,2$, $p$, $q$ and $r$ are propositional variables and $\circ\in\{\wedge,\vee\}$). If the conclusion of a rule contains several sets, vertical bars designate splitting of the branch.

We will have two kinds of rules: the one decomposing labelled formulas and the one decomposing labelled pairs of formulas into labelled pairs of formulas with fewer connectives.
\[(\mathfrak{t}\wedge)\dfrac{\mathfrak{t}[\phi\wedge\chi]}{\mathfrak{t}[\phi],\mathfrak{t}[\chi]}\quad(\mathfrak{t}\vee)\dfrac{\mathfrak{t}[\phi\vee\chi]}{\mathfrak{t}[\phi]\mid\mathfrak{t}[\chi]\mid\mathfrak{m}[\phi],\mathfrak{m}[\chi],\phi\nsim\chi}\quad(\mathfrak{t}\neg)\dfrac{\mathfrak{t}[\neg\phi]}{\mathfrak{f}[\phi]}\]\vspace{.5em}
\[(\mathfrak{m}\wedge)\dfrac{\mathfrak{m}[\phi\wedge\chi]}{\mathfrak{t}[\phi],\mathfrak{m}[\chi]\mid\mathfrak{m}[\phi],\mathfrak{t}[\chi]\mid\mathfrak{m}[\phi],\mathfrak{m}[\chi],\phi\sim\chi}\]
\[(\mathfrak{m}\vee)\dfrac{\mathfrak{m}[\phi\vee\chi]}{\mathfrak{f}[\phi],\mathfrak{m}[\chi]\mid\mathfrak{m}[\phi],\mathfrak{f}[\chi]\mid\mathfrak{m}[\phi],\mathfrak{m}[\chi],\phi\sim\chi}\]
\[(\mathfrak{m}\neg)\dfrac{\mathfrak{m}[\neg\phi]}{\mathfrak{m}[\phi],\phi\sim\neg\phi}\]\vspace{.5em}
\[(\mathfrak{f}\wedge)\dfrac{\mathfrak{f}[\phi\wedge\chi]}{\mathfrak{f}[\phi]\mid\mathfrak{f}[\chi]\mid\mathfrak{m}[\phi],\mathfrak{m}[\chi],\phi\nsim\chi}\quad(\mathfrak{f}\vee)\dfrac{\mathfrak{f}[\phi\vee\chi]}{\mathfrak{f}[\phi],\mathfrak{f}[\chi]}\quad(\mathfrak{f}\neg)\dfrac{\mathfrak{f}[\neg\phi]}{\mathfrak{t}[\phi]}\]\vspace{.5em}
\[(\sim\mathsf{sym})\dfrac{\phi\sim\chi}{\chi\sim\phi}\quad(\nsim\mathsf{sym})\dfrac{\phi\nsim\chi}{\chi\nsim\phi}\quad(\sim\mathsf{trans})\dfrac{p\sim q,q\sim r}{p\sim r}\quad(\nsim\mathsf{trans})\dfrac{p\nsim q,q\sim r}{p\nsim r}\]\vspace{.5em}
\[(\neg\!\sim\!)\dfrac{\neg\phi\!\sim\!\chi}{\phi\!\sim\!\chi}\quad(\neg\!\nsim\!)\dfrac{\neg\phi\!\nsim\!\chi}{\phi\!\nsim\!\chi}\quad(\circ\!\sim\!)\dfrac{\phi\!\sim\!\chi_1\!\circ\!\chi_2,\mathfrak{m}[\chi_i]}{\phi\!\sim\!\chi_i}\quad(\circ\!\nsim\!)\dfrac{\phi\!\nsim\!\chi_1\!\circ\!\chi_2,\mathfrak{m}[\chi_i]}{\phi\nsim\chi_i}\]
We will further present our tableaux proofs in a tree-like form with edges designating the splitting of the branch. When talking about branches we will also represent them with sets and therefore will not repeat labelled formulas or pairs thereof if they have already occurred up the branch.  Moreover, we will omit the use of rules for symmetricity of $\sim$ and $\nsim$ for the sake of brevity.
		
We say that a branch of a tableau is \textit{closed} iff it contains one of the following.
\begin{enumerate}
\item $\mathcal{M}[\phi]$ and $\mathcal{M}'[\phi]$ with $\mathcal{M}\neq\mathcal{M}'$.
\item $\phi\sim\chi$ and $\phi\nsim\chi$.
\item $\phi\nsim\phi$.
\item\label{condition4} $\phi_1\nsim\phi_2$, $\phi_1\nsim\phi_3$, $\phi_1\nsim\phi_4$, $\phi_2\nsim\phi_3$, $\phi_2\nsim\phi_4$, and $\phi_3\nsim\phi_4$.
\end{enumerate}
		
A branch is called \textit{open}, iff it is not closed. Finally, we call a branch $\mathcal{B}$ \textit{complete}, iff whenever $\mathcal{B}$ contains the premise(s) of an instance of a rule, $\mathcal{B}$ also contains one set of conclusions\footnote{As an example, consider the branch $\{\mathfrak{m}[p\wedge q],\mathfrak{m}[p],\mathfrak{t}[q]\}$. It is complete since the rule $\mathfrak{m}\wedge$ requires splitting and one of its resulting sub-branches is exactly $\{\mathfrak{m}[p],\mathfrak{t}[q]\}$. On the other hand, $\{\mathfrak{t}[p\vee q],\mathfrak{m}[p],\mathfrak{m}[q]\}$ is not complete since the application of $\mathfrak{t}\vee$ forms a sub-branch $\{\mathfrak{m}[p],\mathfrak{m}[q],p\nsim q\}$.}.

We say that a tableau is closed iff all its branches are closed.
		
We say that there is a tableaux proof for $\phi\vdash_{\mathrm{ETL}_{\mathbf{M3}}}\chi$ iff tableaux beginning with $\{\mathfrak{t}[\phi],\mathfrak{m}[\chi]\}$ and $\{\mathfrak{t}[\phi],\mathfrak{f}[\chi]\}$ are \textbf{both} closed.
		
Dually, there is a tableaux proof for $\phi\vdash_{\mathrm{NFL}_{\mathbf{M3}}}\chi$ iff tableaux beginning with $\{\mathfrak{m}[\phi],\mathfrak{f}[\chi]\}$ and $\{\mathfrak{t}[\phi],\mathfrak{f}[\chi]\}$ are \textbf{both} closed.
\end{definition}

The conditions dictating closure of branches mean the following: (1) says that $\phi$ cannot have two different values under one valuation; (2) says that is impossible for two formulas to have equal and distinct valuations simultaneously; (3) says that it is impossible for a value of a formula to be distinct from itself; finally, (4) says that there cannot be a valuation $v$ such that $v(\phi_1),\ldots,v(\phi_4)\in\{\mathbf{B},\mathbf{0},\mathbf{N}\}$ are pairwise distinct.
	
It is also easy to see that for any sequent $\phi\vdash\chi$ its tableau will terminate since our rules allow to decompose each labelled formula into variables and reduce all labelled pairs formulas to labelled pairs of variables.
	
Consider two examples below that show how to construct a tableau and obtain a refuting valuation from its complete open branch. First, we show that there is no valuation $v$ such that $v((p\vee q)\wedge r)=\mathbf{T}$ but $v(p\vee(q\wedge r))\in\{\mathbf{B},\mathbf{0},\mathbf{N}\}$ (Fig.~\ref{goodtableau}). Then we show that $(p\wedge\neg p)\vee(q\wedge\neg q)\nvdash r$ and then extract a refuting valuation (Fig.~\ref{badtableau}).
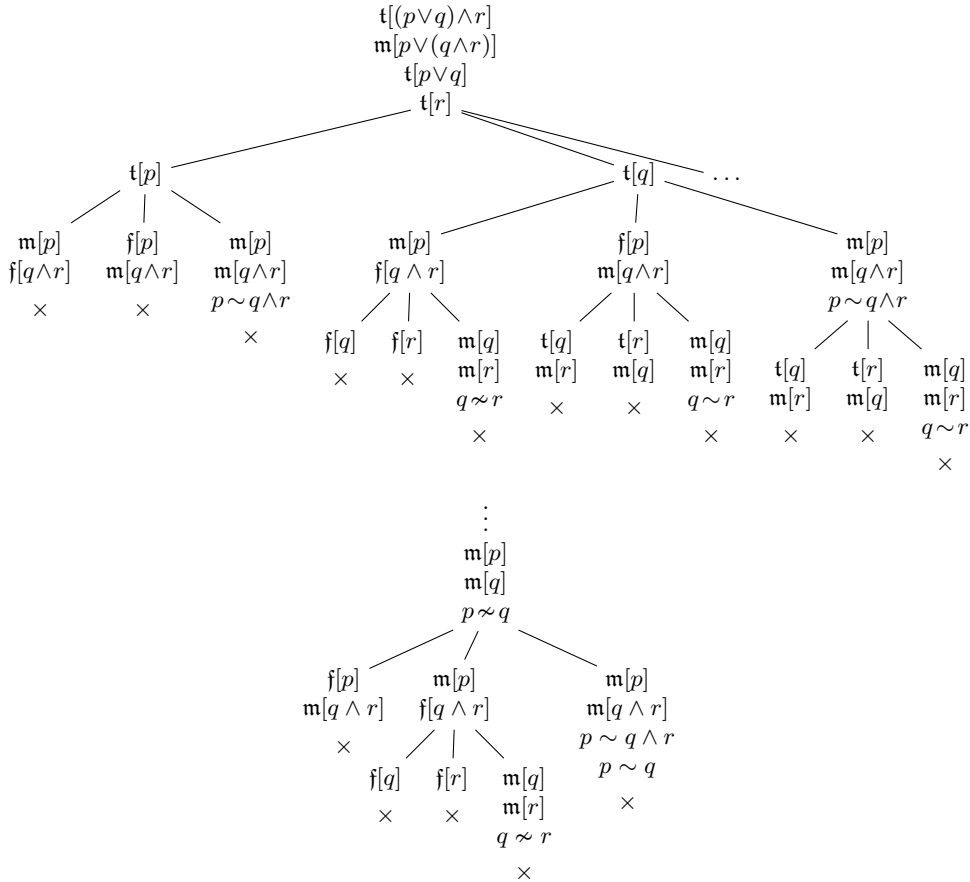
\begin{figure}[h!]
\begin{center}
{\small
\begin{forest}
smullyan tableaux
[{\mathfrak{t}[(p\!\vee\!q)\!\wedge\!r]}
[{\mathfrak{m}[p\!\vee\!(q\!\wedge\!r)]}
[{\mathfrak{t}[p\!\vee\!q]}
[{\mathfrak{t}[r]}
[{\mathfrak{t}[p]}[{\mathfrak{m}[p]}[{\mathfrak{f}[q\!\wedge\!r]}, closed]][{\mathfrak{f}[p]}[{\mathfrak{m}[q\!\wedge\!r]}, closed]][{\mathfrak{m}[p]}[{\mathfrak{m}[q\!\wedge\!r]}[p\!\sim\!q\!\wedge\!r, closed]]]]
[{\mathfrak{t}[q]}[{\mathfrak{m}[p]}[{\mathfrak{f}[q\wedge r]}[{\mathfrak{f}[q]}, closed][{\mathfrak{f}[r]}, closed][{\mathfrak{m}[q]}[{\mathfrak{m}[r]}[{q\!\nsim\!r}, closed]]]]]
[{\mathfrak{f}[p]}[{\mathfrak{m}[q\!\wedge\!r]}[{\mathfrak{t}[q]}[{\mathfrak{m}[r]}, closed]][{\mathfrak{t}[r]}[{\mathfrak{m}[q]}, closed]][{\mathfrak{m}[q]}[{\mathfrak{m}[r]}[{q\!\sim\!r}, closed]]]]]
[{\mathfrak{m}[p]}[{\mathfrak{m}[q\!\wedge\!r]}[p\!\sim\!q\!\wedge\!r[{\mathfrak{t}[q]}[{\mathfrak{m}[r]}, closed]][{\mathfrak{t}[r]}[{\mathfrak{m}[q]}, closed]][{\mathfrak{m}[q]}[{\mathfrak{m}[r]}[{q\!\sim\!r}, closed]]]]]]]
[\ldots]]]]]
\end{forest}
\begin{forest}
smullyan tableaux
[\vdots[{\mathfrak{m}[p]}[{\mathfrak{m}[q]}[{p\!\nsim\!q}[{\mathfrak{f}[p]}[{\mathfrak{m}[q\wedge r]}, closed]][{\mathfrak{m}[p]}[{\mathfrak{f}[q\wedge r]}[{\mathfrak{f}[q]}, closed][{\mathfrak{f}[r]}, closed][{\mathfrak{m}[q]}[{\mathfrak{m}[r]}[q\nsim r, closed]]]]][{\mathfrak{m}[p]}[{\mathfrak{m}[q\wedge r]}[p\sim q\wedge r[p\sim q, closed]]]]]]]]
\end{forest}}
\end{center}
\caption{A tableaux showing that there is no valuation $v$ such that $v((p\vee q)\wedge r)=\mathbf{T}$ but $v(p\vee(q\wedge r))\in\{\mathbf{B},\mathbf{0},\mathbf{N}\}$. All branches are closed.}\label{goodtableau}
\end{figure}
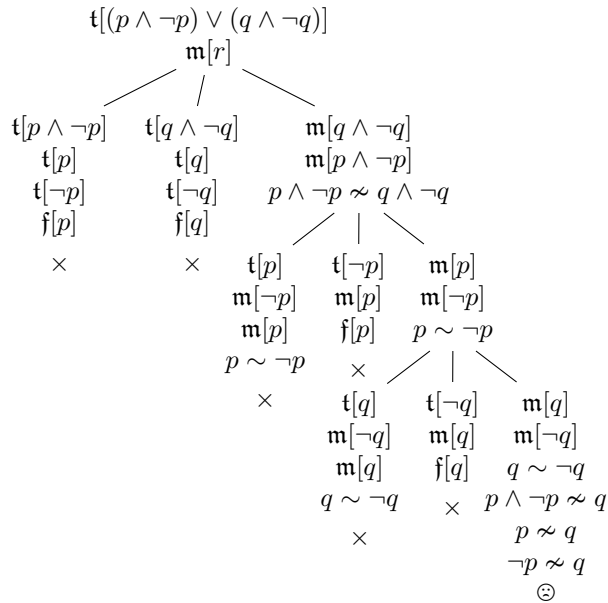
\begin{figure}[h!]
\begin{center}
\begin{forest}
smullyan tableaux
[{\mathfrak{t}[(p\wedge\neg p)\vee(q\wedge\neg q)]}[{\mathfrak{m}[r]}
[{\mathfrak{t}[p\wedge\neg p]}[{\mathfrak{t}[p]}[{\mathfrak{t}[\neg p]}[{\mathfrak{f}[p]}, closed]]]]
[{\mathfrak{t}[q\wedge\neg q]}[{\mathfrak{t}[q]}[{\mathfrak{t}[\neg q]}[{\mathfrak{f}[q]}, closed]]]]
[{\mathfrak{m}[q\wedge\neg q]}[{\mathfrak{m}[p\wedge\neg p]}[p\wedge\neg p\nsim q\wedge\neg q[{\mathfrak{t}[p]}[{\mathfrak{m}[\neg p]}[{\mathfrak{m}[p]}[p\sim\neg p, closed]]]][{\mathfrak{t}[\neg p]}[{\mathfrak{m}[p]}[{\mathfrak{f}[p]}, closed]]][{\mathfrak{m}[p]}[{\mathfrak{m}[\neg p]}[p\sim\neg p[{\mathfrak{t}[q]}[{\mathfrak{m}[\neg q]}[{\mathfrak{m}[q]}[q\sim\neg q, closed]]]][{\mathfrak{t}[\neg q]}[{\mathfrak{m}[q]}[{\mathfrak{f}[q]}, closed]]][{\mathfrak{m}[q]}[{\mathfrak{m}[\neg q]}[q\sim\neg q[p\wedge\neg p\nsim q[p\nsim q[\neg p\nsim q[\frownie]]]]]]]]]]]]]]]
\end{forest}
\end{center}
\caption{A tableau refuting $(p\wedge\neg p)\vee(q\wedge\neg q)\nvdash r$. The frownie shows a complete open branch.}\label{badtableau}
\end{figure}
	
A refuting valuation is obtained as follows. First, we note that all variables --- $p$, $q$ and $r$ --- are labelled $\mathfrak{m}$. So, we need to assign them values from $\{\mathbf{B},\mathbf{N},\mathbf{0}\}$. We also see that $p\nsim q$, hence $p$ and $q$ should have distinct values, but there is no such label for $r$. So, let us set $v(p)=\mathbf{B}$, $v(q)=\mathbf{0}$ and $v(r)=\mathbf{B}$. It is easy to see that $v((p\wedge\neg p)\vee(q\wedge\neg q))=\mathbf{T}$ but $v(r)=\mathbf{B}$.
	
We draw the reader's attention to the fact that \textbf{B}, \textbf{0} and \textbf{N} are indistinguishable w.r.t.\ $\vDash_{\mathrm{ETL}_{\mathbf{M3}}}$ and $\vDash_{\mathrm{NFL}_{\mathbf{M3}}}$ in the sense that either all of them belong to the set of designated values (as in the case of $\mathrm{NFL}_{\mathbf{M3}}$) or none of them do ($\mathrm{ETL}_{\mathbf{M3}}$). This allows us to use only three labels for formulas and two for pairs of formulas.

The last fact we wish to notice in this section is that two more general rules, namely, $(\sim\!\mathsf{Tr.}\!)\dfrac{\{\phi\sim\chi,\chi\sim\psi\}}{\{\phi\sim\psi\}}$ and $(\nsim\!\mathsf{Tr.}\!)\dfrac{\{\phi\nsim\chi,\chi\sim\psi\}}{\{\phi\sim\psi\}}$ are admissible in the following sense.
\begin{lemma}\label{cutadmissibility}
For any formulas $\phi$, $\chi$ and $\psi$ the following holds: \emph{any branch $\mathcal{B}$ containing (1) $\phi\sim\chi$, $\chi\sim\psi$ and $\phi\nsim\psi$ or (2) $\phi\nsim\chi$, $\chi\sim\psi$ and $\phi\sim\psi$ can be closed.}
\end{lemma}
\begin{proof}
We prove this by induction on complexity of $\phi$, $\chi$ and $\psi$. The basis (when $\phi$, $\chi$ and $\psi$ are propositional variables) holds by virtue of the definition of tableaux. We will further prove only the case (1) since (2) can be proved in the same way.

Now, let $\phi$, $\chi$, and $\psi$ be arbitrary formulas. Say, the branch contains $\phi\sim\chi$, $\chi\sim \psi$ and $\phi\nsim \psi$. We have three cases for each $\tau\in\{\phi,\chi,\psi\}$. Either $\tau=\tau_1\wedge\tau_2$, $\tau=\tau_1\vee\tau_2$, or $\tau=\neg\tau'$. We will consider only the case when each $\tau$ is $\tau_1\wedge\tau_2$ (other cases are tackled similarly).

If $\mathcal{B}$ contains $\phi_1\wedge\phi_2\sim\chi_1\wedge\chi_2$ and $\chi_1\wedge\chi_2\sim\psi_1\wedge\psi_2$, then $\mathfrak{m}[\phi_1\wedge\phi_2]$, $\mathfrak{m}[\chi_1\wedge\chi_2]$, $\mathfrak{m}[\psi_1\wedge\psi_2]$ are also present (since a~pair of formulas labelled with $\sim$ can appear only if both of them are labelled~$\mathfrak{m}$). But then for each $\tau\in\{\phi,\chi,\psi\}$ either $\mathfrak{m}[\tau_1]$, or $\mathfrak{m}[\tau_2]$, or both are in $\mathcal{B}$. Hence, we can infer $\phi_i\sim\chi_j$, $\chi_j\sim\psi_k$ and $\phi_i\nsim\psi_k$ for some $i,j,k\in\{1,2\}$ using $(\circ\sim)$ rule. By induction hypothesis, this branch can be closed.
\end{proof}

Note also that with the rules $\sim\mathsf{Tr.}$ and $\nsim\mathsf{Tr.}$ from Lemma~\ref{cutadmissibility} we do not need rules $\neg\sim$ and $\neg\nsim$.
\begin{lemma}\label{cutisstrong}
Assume, we have rules $\sim\mathsf{Tr.}$ and $\nsim\mathsf{Tr.}$. Then if a branch contains $\neg\phi\sim\chi$, it can be extended with $\phi\sim\chi$ (likewise for $\phi\nsim\chi$ and $\phi\nsim\chi$).
\end{lemma}
\begin{proof}
Indeed, if a branch contains $\neg\phi\sim\chi$, it also contains $\mathfrak{m}[\neg\phi]$. From here we infer $\mathfrak{m}[\phi]$ and $\phi\sim\neg\phi$. Now, by $\sim\mathsf{Tr.}$ and $\sim\mathsf{trans}$ rules, we obtain $\phi\sim\chi$.

The proof of the $\neg\nsim$ case is the same.
\end{proof}
\subsection{Correctness and completeness}\label{nDETLadequacy}
Observe that in our examples on both Fig.~\ref{goodtableau} and~\ref{badtableau} we tested only whether there is a valuation $v$ such that for a sequent $\phi\vdash\chi$ $v(\phi)=\mathbf{T}$ and $v(\chi)\in\{\mathbf{B},\mathbf{0},\mathbf{N}\}$. So, a case of a valuation $v'$ such that $v'(\phi)=\mathbf{T}$ and $v'(\chi)=\mathbf{F}$ was left unconsidered. While for the example in Fig.~\ref{badtableau} considering only one of two options was enough to refute the sequent, to prove a given sequent we have to tackle both of them.
	
Before we start proving the correctness and completeness of our tableaux, we need a definition of ‘realisation’ which we adapt from~\cite{DAgostino1990}.
\begin{definition}[Realisation]
We say that a valuation $v$ realises a labelled formula or a labelled pair of formulas in the following cases.
\begin{enumerate}
\item $\mathfrak{t}[\phi]$ and $v(\phi)=\mathbf{T}$.
\item $\mathfrak{m}[\phi]$ and $v(\phi)\in\{\mathbf{B},\mathbf{0},\mathbf{N}\}$.
\item $\mathfrak{f}[\phi]$ and $v(\phi)=\mathbf{F}$.
\item $\phi\sim\chi$ and $v(\phi)=v(\chi)$ and $v(\phi),v(\chi)\in\{\mathbf{B},\mathbf{0},\mathbf{N}\}$.
\item $\phi\nsim\chi$ and $v(\phi)\neq v(\chi)$ and $v(\phi),v(\chi)\in\{\mathbf{B},\mathbf{0},\mathbf{N}\}$.
\end{enumerate}
		
We say that a branch of a tableau is realised iff its every node is realised.
\end{definition}
\begin{theorem}[Correctness for $\mathrm{ETL}_{\mathbf{M3}}$]\label{M3correctness}
If there is a tableaux proof for $\phi\vdash_{\mathrm{ETL}_{\mathbf{M3}}}\chi$, then $\phi\vDash_{\mathrm{ETL}_{\mathbf{M3}}}\chi$.
\end{theorem}
\begin{proof}
It is easy to show that if premise(s) of a rule is realisable, then so is at least one of its conclusions. Finally, it is evident that no valuation can realise a closed branch. This concludes the proof.
\end{proof}
	
To prove completeness theorem we adapt the technique from~\cite[Theorem~5.4.1]{DAgostino1990}.
\begin{theorem}[Completeness for $\mathrm{ETL}_{\mathbf{M3}}$]\label{M3completeness}
If $\phi\vDash_{\mathrm{ETL}_{\mathbf{M3}}}\chi$, then there is a tableaux proof for $\phi\vdash_{\mathrm{ETL}_{\mathbf{M3}}}\chi$.
\end{theorem}
\begin{proof}
We prove by contraposition: we will show that each complete open branch is realisable. Now, having a complete open branch $\mathcal{B}$ assign to each labelled variable $p$ occurring in $\mathcal{B}$ a value from $\{\mathbf{T},\mathbf{B},\mathbf{0},\mathbf{N},\mathbf{F}\}$ as follows.
\begin{enumerate}
\item If $\mathfrak{t}[p]\in\mathcal{B}$, then $v(p)=\mathbf{T}$.
\item If $\mathfrak{f}[p]\in\mathcal{B}$, then $v(p)=\mathbf{F}$.
\item If $\mathfrak{m}[p]\in\mathcal{B}$, we proceed as follows.
\begin{enumerate}
\item First, find all variables labelled $\mathfrak{m}$ and order them alphabetically (we will denote such variables $q_1,\ldots,q_n$).
\item Set $v(q_1)=\mathbf{B}$. Then find all $q_i$ such that $q_i\nsim q_1\notin\mathcal{B}$ and set $v(q_i)=\mathbf{B}$ for each of them.
\item Now let $q_j$ be the first variable such that $q_j\nsim q_1$ occurs in $\mathcal{B}$. Set $v(q_j)=\mathbf{0}$. Then find all $q_{j'}$ such that $q_{j'}\nsim q_j$ does not occur in $\mathcal{B}$. Set $v(q_{j'})=\mathbf{0}$ for each of them.
\item Now let $q_k$ be the first variable such that $q_k\nsim q_1$ and $q_k\nsim q_j$ are in $\mathcal{B}$. Set $v(q_k)=\mathbf{N}$. Then find all $q_{k'}$ such that $q_{k'}\nsim q_k\notin\mathcal{B}$. Set $v(q_{k'})=\mathbf{N}$ for each of them.
%\item Now let $q_{l_1}$, \ldots, $q_{l_m}$ be all variables labelled $m$ such that for any $q_{l_i}$ there is no $q_r$ ($r\leqslant n$) for which $q_r\sim q_{l_i}$ or $q_r\nsim q_{l_i}$. Set $v(q_{l_i})=\mathbf{B}$ for each of these. 
%\item For each $q_s$ such  that $q_s\nsim q_j$, $q_s\sim q_j$, $q_s\nsim q_k$ and $q_s\sim q_k$ do not occur in $\mathcal{B}$ but $q_s\nsim q_1$ does occur in $\mathcal{B}$ set $v(q_s)=\mathbf{0}$.
%\item For each $q_t$ such that $q_t\nsim q_1$, $q_s\sim q_1$, $q_t\nsim q_k$ and $q_t\sim q_k$ do not occur in $\mathcal{B}$ but $q_t\nsim q_j$ does occur in $\mathcal{B}$ set $v(q_t)=\mathbf{N}$.
%\item Finally, for each $q_u$ such that $q_t\nsim q_1$, $q_s\sim q_1$, $q_t\nsim q_j$ and $q_t\sim q_j$ do not occur in $\mathcal{B}$ but $q_t\nsim q_k$ does occur in $\mathcal{B}$ set $v(q_t)=\mathbf{B}$.
\end{enumerate}
\end{enumerate}
		
Notice that it is the rules $\sim\mathsf{trans}$ and $\nsim\mathsf{trans}$ that enable us to search for all labelled variables that take the same (or different) value as the one we took first. Notice also that due to Lemma~\ref{cutadmissibility} $\mathcal{B}$ cannot contain $\phi\sim\chi$, $\chi\sim\psi$ and $\phi\nsim\psi$ (or $\phi\nsim\chi$, $\chi\sim\psi$ and $\phi\sim\psi$), lest it would be closed (and, hence, unrealisable).
		
Observe now that $v$ defined as above realises all labelled variables and labelled pairs of variables occurring in~$\mathcal{B}$. Then note that for each $\alpha$-rule decomposing labelled formulas if $v$ realises all its conclusions, $v$ also realises its premise. Further, for each $\beta$-rule decomposing labelled formulas if $v$ realises all its conclusions in one of its resulting branches, $v$ also realises the premise.
		
In $\sim\mathsf{trans}$ and $\nsim\mathsf{trans}$ we have by construction of $v$ that $v$ realises both premises if it realises the conclusion.
		
Other rules dealing with labelled pairs of formulas are invertible in the sense that if $v$ realises the conclusion, it also realises the premises.
		
In particular, consider the $\circ\sim$ rule with $\circ=\vee$. By induction hypothesis, $v$ realises $\phi\sim\chi_1$ and we also have $\mathfrak{m}[\chi_1]$ and $\phi\sim\chi_1\vee\chi_2$ in $\mathcal{B}$. Since $\mathcal{B}$ is complete, we also have $\mathfrak{m}[\phi]$ and $\mathfrak{m}[\chi_1\vee\chi_2]$. Furthermore, $v$ realises $\mathfrak{m}[\phi]$ and $\mathfrak{m}[\chi_1]$. Again, by completeness of $\mathcal{B}$, we also have either $\mathfrak{m}[\chi_2]$ and $\chi_1\sim\chi_2$ or $\mathfrak{f}[\chi_2]$. In the first case we obtain that $v$ realises $\chi_1\sim\chi_2$ by induction hypothesis and in the second case we obtain that $v$ realises $\mathfrak{f}[\chi_2]$. In both cases, it is straightforward to see that $v$ realises $\phi\sim\chi_1\vee\chi_2$.
		
The result follows by an induction on the complexity of labelled formulas and pairs of formulas in $\mathcal{B}$.
\end{proof}
The proofs of the following two theorems are the same as proofs of Theorems~\ref{M3correctness} and~\ref{M3completeness}, respectively.
\begin{theorem}[Correctness for $\mathrm{NFL}_{\mathbf{M3}}$]\label{M3NFLcorrectness}
If there is a tableaux proof for $\phi\vdash_{\mathrm{NFL}_{\mathbf{M3}}}\chi$, then $\phi\vDash_{\mathrm{NFL}_{\mathbf{M3}}}\chi$.
\end{theorem}
\begin{theorem}[Completeness for $\mathrm{NFL}_{\mathbf{M3}}$]\label{M3NFLcompleteness}
If $\phi\vDash_{\mathrm{NFL}_{\mathbf{M3}}}\chi$, then there is a tableaux proof for $\phi\vdash_{\mathrm{NFL}_{\mathbf{M3}}}\chi$.
\end{theorem}
\section{Generalisation}\label{generalisation}
In this section we will show how to generalise our tableaux system on ETL- and NFL-like logics built upon arbitrary \textbf{Mn} lattices.
\subsection{Tuning the tableaux}\label{tuning}
Observe that \textbf{Mn} lattices differ only in the number of pairwise incomparable w.r.t.\ ‘upwards’ order elements, namely, there are $n$ such elements for each lattice (except for the three-element one where any two elements are comparable). Since we will always use only the labels and rules from Definition~\ref{M3tableaux}, in contrast to the traditional approach from~\cite{Salzer1996CADE,BaazFermullerSalzerZach1996,BaazFermullerSalzer2000} and~\cite{MarcelinoCaleiro2019} which demands the number of labels and rules depending on the number of truth values, we can adjust our tableaux as follows.
\begin{definition}[Tableaux for ETL- and NFL-like logics on \textbf{Mn} lattices]\label{Mntableaux}
We define tableaux and all rules for them as in Definition~\ref{M3tableaux}. The only difference is in condition~(\ref{condition4}) for branch closure. Namely, for each \textbf{Mn} lattice set
\begin{enumerate}
\item[($4'$)\label{condition4'}] There are $\phi_1$, \ldots, $\phi_{n+1}$ in the branch such that for all $i,j\leqslant n+1$ if $i\neq j$, then $\phi_i\nsim\phi_j$\footnote{For LP and Kleene's logic this reduces to the following: the branch is closed if it contains $\phi\nsim\chi$.}.
\end{enumerate}
		
The notion of a tableaux proof for $\phi\vdash_{\mathrm{ETL}_{\mathbf{Mn}}}\chi$ and $\phi\vdash_{\mathrm{NFL}_{\mathbf{Mn}}}\chi$ also coincides with that from Definition~\ref{M3tableaux}.
\end{definition}
	
The following theorems are easy generalisations of Theorems~\ref{M3correctness}, \ref{M3completeness}, \ref{M3NFLcorrectness} and~\ref{M3NFLcompleteness}.
\begin{theorem}\label{ETLMncompleteness}
For any lattice $\mathbf{Mn}$ and its ETL-like logic $\phi\vDash_{\mathrm{ETL}_{\mathbf{Mn}}}\chi$ iff there is a~tableaux proof of $\phi\vdash_{\mathrm{ETL}_{\mathbf{Mn}}}\chi$.
\end{theorem}
\begin{theorem}\label{NFLMncompleteness}
For any lattice $\mathbf{Mn}$ and its NFL-like logic $\phi\vDash_{\mathrm{NFL}_{\mathbf{Mn}}}\chi$ iff there is a~tableaux proof of $\phi\vdash_{\mathrm{NFL}_{\mathbf{Mn}}}\chi$.
\end{theorem}
\subsection{Separating $\mathbf{Mn}$ and $\mathbf{Mn+1}$}
In the previous section we have provided an easy ‘tuning’ to our tableaux. It is now natural to ask which ‘formula-formula’ consequences will be lost when transitioning from ETL- or NFL-like logic built upon $\mathbf{Mn}$ to that upon $\mathbf{Mn+1}$. Note also, that $\mathbf{Mn}$ is a submatrix of $\mathbf{Mn+1}$, so we have the following decreasing sequence of ETL- and NFL-like logics with $L_1\supseteq L_2$ (read ‘$L_1$ contains $L_2$’ or ‘$L_2$ is included into $L_1$’) meaning that each consequence valid in $L_2$ is also valid in $L_1$.
\begin{eqnarray}
\mathrm{ETL}_{\mathbf{M1}}=\mathbf{K3}\supsetneq\mathrm{ETL}_{\mathbf{M2}}=\mathrm{ETL}\supsetneq\mathrm{ETL}_{\mathbf{M3}}\supsetneq\mathrm{ETL}_{\mathbf{M4}}\supseteq\ldots\label{ETLsequence}
\\
\mathrm{NFL}_{\mathbf{M1}}=\mathrm{LP}\supsetneq\mathrm{NFL}_{\mathbf{M2}}=\mathrm{NFL}\supsetneq\mathrm{NFL}_{\mathbf{M3}}\supsetneq\mathrm{NFL}_{\mathbf{M4}}\supseteq\ldots\label{NFLsequence}
\end{eqnarray}

We know from~\eqref{ETLM3neqETLM4} that ETL-like logic built upon \textbf{M3} strictly contains \textbf{M4}. Dualisation of~\eqref{ETLM3neqETLM4} will produce a consequence valid in $\mathrm{NFL}_{\mathbf{M3}}$ but not in $\mathrm{NFL}_{\mathbf{M4}}$.

We now want to convince ourselves that all other inclusions in~\eqref{ETLsequence} and~\eqref{NFLsequence} are also strict. Indeed, a family of consequences separating $\mathbf{Mn}$ and $\mathbf{Mn+1}$ for each $\mathbf{n}$ is easy to produce. Consider the following consequences.
\begin{eqnarray}
%\mathscr{D}_{n}\coloneqq\left(p_1\!\vee\!\bigvee\limits_{\begin{matrix}i,j,k\!\in\!\binom{n+1}{3}\end{matrix}}(p_i\!\wedge(p_j\!\vee\!p_k))\right)\!\wedge\!\bigwedge\limits_{l=2}^{n+1}(p_1\!\vee\!p_l)\vdash\nonumber\\\vdash\!p_1\!\vee\!\bigvee\limits_{\begin{matrix}i,j\!\in\binom{\!n\!+\!1}{2}\end{matrix}}(p_i\!\wedge\!p_j)
\mathscr{D}_{n}\coloneqq \bigwedge\limits_{i=2}^{n+1}(p_1\!\vee\!p_i)\vdash\!p_1\!\vee\!\bigvee\limits_{\begin{matrix}i,j\!\in\binom{\!n\!+\!1}{2}\end{matrix}}(p_i\!\wedge\!p_j)\label{MnMn+1separation}
\end{eqnarray}

It is easy to check that for any $n>2$ $\mathscr{D}_n$ is valid in $\mathrm{ETL}_{\mathbf{Mn}}$ but is not valid in $\mathrm{ETL}_{\mathbf{Mn+1}}$ (to obtain a refuting valuation set $v(p_i)=\mathbf{i}$ with $1\leqslant i\leqslant n$). Again, by dualisation of~\eqref{MnMn+1separation}, we obtain a family of ‘formula-formula’ consequences separating $\mathrm{NFL}_{\mathbf{Mn}}$ from $\mathrm{NFL}_{\mathbf{Mn+1}}$.

Moreover, it is easy to see that the only refuting valuation for~\ref{MnMn+1separation} is some injective map from variables to the ‘middle’ values of $\mathbf{Mn+1}$. Thus, only the tableau starting with $\left\{\mathfrak{t}\!\left[\bigwedge\limits_{i=2}^{n+1}(p_1\!\vee\!p_i)\right],\mathfrak{m}\!\left[p_1\!\vee\!\!\!\!\bigvee\limits_{\scriptsize{\begin{matrix}i,j\!\in\binom{\!n\!+\!1}{2}\end{matrix}}}(p_i\!\wedge\!p_j)\right]\right\}$\footnote{On the other hand, the tableau beginning with $\left\{\mathfrak{t}\left[\bigwedge\limits_{i=2}^{n+1}(p_1\!\vee\!p_i)\right],\mathfrak{f}\left[p_1\!\vee\!\bigvee\limits_{\scriptsize{\begin{matrix}i,j\!\in\binom{\!n\!+\!1}{2}\end{matrix}}}(p_i\!\wedge\!p_j)\right]\right\}$ will be closed.} will have a branch containing $p_i\nsim p_j$ for all $1\leqslant i<j\leqslant n+1$ which will be open, while all other branches will be closed. It is also clear that since an open branch in a tableau for $\mathrm{ETL}_{\mathbf{Mn}}$ ($\mathrm{NFL}_{\mathbf{Mn}}$) is a set (and as such has no repeating entries) and can contain at most $O\left(\binom{n}{2}\right)$ pairs of variables labelled with $\nsim$ or $\sim$, such binary labels do not seriously affect the efficiency of our tableaux.

As expected, the dualisation of~\eqref{MnMn+1separation} will produce a family of consequences separating NFL-like logics on $\mathbf{Mn}$ and $\mathbf{Mn+1}$ for $n>2$.
\subsection{$\mathbf{M\omega}$}
Until now we have considered only finite $\mathbf{Mn}$ lattices and logics built upon them whose corresponding tableaux differed only in one closure rule~($4'$) which said how many formulas pairwise incomparable w.r.t.\ lattice order there can be. The next logical step is to simply abandon that rule and allow a branch of a~tableau to contain any finite number of pairwise incomparable formulas.

\begin{definition}[Tableaux for ETL- and NFL-like logics on $\mathbf{M}\omega$]
The rules and the definitions of tableaux remain from definitions~\ref{M3tableaux} and~\ref{Mntableaux}. The only change is the following one.

We say that a branch of a tableau is \textit{closed} iff it contains one of the following.
\begin{enumerate}
\item $\mathcal{M}[\phi]$ and $\mathcal{M}'[\phi]$ with $\mathcal{M}\neq\mathcal{M}'$.
\item $\phi\sim\chi$ and $\phi\nsim\chi$.
\item $\phi\nsim\phi$.
\end{enumerate}
\end{definition}

One would expect such tableaux to formalise ETL- and NFL-like logics on the following lattice which we will call $\mathbf{M}\omega$.
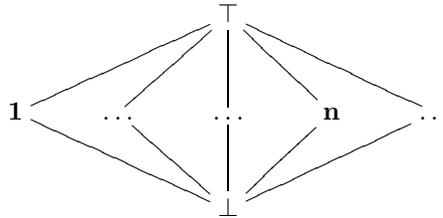
\[\xymatrix{
&&\top\ar@{-}[dll]\ar@{-}[dl]\ar@{-}[d]\ar@{-}[drr]\ar@{-}[dr]&&\\
\mathbf{1}&\ldots&\ldots&\mathbf{n}&\ldots\\
&&\bot\ar@{-}[ull]\ar@{-}[ul]\ar@{-}[u]\ar@{-}[urr]\ar@{-}[ur]&&
}\]
This is indeed the case as the following theorems show.
\begin{theorem}
$\phi\vDash_{\mathrm{ETL}_{\mathbf{M}\omega}}\chi$ iff there is a tableaux proof of $\phi\vdash_{\mathrm{ETL}_{\mathbf{M}\omega}}\chi$.
\end{theorem}
\begin{theorem}
$\phi\vDash_{\mathrm{NFL}_{\mathbf{M}\omega}}\chi$ iff there is a tableaux proof of $\phi\vdash_{\mathrm{NFL}_{\mathbf{M}\omega}}\chi$.
\end{theorem}
The theorems can be easily proved if we construct a~realising valuation for a~complete open branch $\mathcal{B}$ as follows.
\begin{enumerate}
\item If $\mathfrak{t}[p]\in\mathcal{B}$, then $v(p)=\top$.
\item If $\mathfrak{f}[p]\in\mathcal{B}$, then $v(p)=\bot$.
\item If $\mathfrak{m}[p]\in\mathcal{B}$, we proceed as follows.
\begin{enumerate}
\item First, find all variables labelled $\mathfrak{m}$ and order them alphabetically (we will denote such variables $q_1,\ldots,q_n$).
\item Set $v(q_1)=\mathbf{1}$. Then find all $q_i$ such that $q_i\nsim q_1\notin\mathcal{B}$ and set $v(q_i)=\mathbf{1}$ for each of them.
\item Now let $q_j$ be the first variable such that $q_j\nsim q_1$ occurs in $\mathcal{B}$. Set $v(q_j)=\mathbf{2}$. Then find all $q_{j'}$ such that $q_{j'}\nsim q_j$ does not occur in $\mathcal{B}$. Set $v(q_{j'})=\mathbf{2}$ for each of them.
\item Proceed until all variables labelled $\mathfrak{m}$ are valuated with $\mathbf{1},\ldots,\mathbf{n},\ldots<\omega$.
\end{enumerate}
\end{enumerate}

Now, it is obvious that none of the consequences~\eqref{MnMn+1separation} are valid in $\mathrm{ETL}_{\mathbf{M}\omega}$. Still, there is a question to be answered: is $\mathrm{ETL}_{\mathbf{M}\omega}$ the ETL-like logic over all bounded lattices from Definition~\ref{ETLNFLlike}?

Note that not every bounded lattice can be used to construct a matrix for an ETL- or NFL-like logic. E.g., De Morgan laws and double negation cannot simultaneously hold in the following lattice. This is why we consider here only the matrices subject to conditions of Definition~\ref{ETLNFLlike}. 
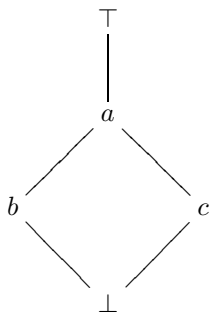
\[\xymatrix{
&\top\ar@{-}[d]&\\
&a\ar@{-}[dl]\ar@{-}[dr]&\\
b&&c\\
&\bot\ar@{-}[ul]\ar@{-}[ur]&
}\]
The following theorems provide an affirmative answer to our question.
\begin{theorem}\label{ETLomega}
Let $\mathrm{ETL}_{\mathfrak{L}}$ be an ETL-like logic over bounded lattice $\mathfrak{L}$ and $\phi\nvDash_{\mathrm{ETL}_{\mathfrak{L}}}\chi$. Then $\phi\nvDash_{\mathrm{ETL}_{\mathbf{M}\omega}}\chi$.
\end{theorem}
\begin{proof}
Let us briefly sketch the proof.
	
It suffices to show that for any valuation $v$ in $\mathfrak{L}$ there is a valuation $v'$ in $\mathbf{M}\omega$ such that $v(\phi)=\top$ iff $v'(\phi)=\top$.

We construct $v'$ as follows.
\begin{itemize}
\item If $v(p)=\top$ or $v(p)=\bot$, then $v'(p)=\top$ or $v'(p)=\bot$, respectively.
\item Otherwise, $v'(p)\in\{\mathbf{1},\ldots,\mathbf{n},\ldots\}$ and the following conditions apply.
	\begin{itemize}
	\item If $v(p\vee q)=\top$, then $v'(p)\neq v'(q)$.
	\item If $v(p\vee q)\neq\top$, then $v'(p)=v'(q)$.
	\end{itemize}
\end{itemize}

Now we can show by induction on $\phi_1$ and $\phi_2$ that for any $\phi_1$ and $\phi_2$ $v(\phi_1\vee\phi_2)=\top$ but $v(\phi_1)\neq\top$ and $v(\phi_2)\neq\top$ iff $v'(\phi_1),v'(\phi_2)\in\{\mathbf{1},\ldots,\mathbf{n},\ldots\}$ and $v'(\phi_1)\neq v'(\phi_2)$. The details are left to the reader.

It is now easy to show by induction on $\phi$ that $v(\phi)=\top$ iff $v'(\phi)=\top$. Again, we leave it to the reader to complete the proof.
\end{proof}
\begin{theorem}\label{NFLomega}
Let $\mathrm{NFL}_{\mathfrak{L}}$ be an NFL-like logic over bounded lattice $\mathfrak{L}$ and $\phi\nvDash_{\mathrm{NFL}_{\mathfrak{L}}}\chi$. Then $\phi\nvDash_{\mathrm{NFL}_{\mathbf{M}\omega}}\chi$.
\end{theorem}
\begin{proof}
Analogously to the proof of Theorem~\ref{ETLomega}.
\end{proof}
\subsection{Putting tableaux into context}
Now, as we have laid out our tableaux and shown their most important properties, it is time to put our tableau systems into a~wider context. First, observe, that rules $\sim\mathsf{trans}$ and $\nsim\mathsf{trans}$ as well as $\sim\mathsf{Tr.}$ and $\nsim\mathsf{Tr.}$ bear some similarity to the so called ‘analytic cut rule’ from \textbf{KE} systems proposed and studied by Mondadori and d'Agostino in~\cite{Mondadori1988,DAgostino1990,DAgostino1992,DAgostinoMondadori1994} (cf. more recent \textbf{KE} and related systems for non-classical logics in~\cite{DAgostino1990,NetoFinger2007,CaleiroMarcos2009,CaleiroMarcosVolpe2015}) in the sense that they are crucial for completeness and thus cannot be eliminated.

Second, rules $\circ\sim$ and $\circ\nsim$ are akin to $B$-rules from $\mathbf{RE}_{\mathrm{fde}}$ system of d'Agostino (cf.~\cite[p.~105, ff.]{DAgostino1990}). Furthermore, the following statement of the subformula property obviously holds for our tableaux.
\begin{proposition}
Every closed tableau for a set $\Gamma$ contains only formulas that are subformulas of the formulas from $\Gamma$.
\end{proposition}

As it is known, \textbf{KE} systems provide an improvement w.r.t.\ both truth tables and ordinary analytic tableaux and are as efficient as natural deduction. Two questions thus arise: first, whether our tableaux are indeed improvement upon truth tables\footnote{Which seems rather unlikely, since our tableaux lack direct analogues of Mondadori's and d'Agostino's analytic cut, namely \[\dfrac{}{\{\mathfrak{t}[\phi]\}\mid\{\mathfrak{m}[\phi]\}\mid\{\mathfrak{f}[\phi]\}}\quad\dfrac{\{\mathfrak{m}[\phi],\mathfrak{m}[\chi]\}}{\{\phi\sim\chi\}\mid\{\phi\nsim\chi\}}\]}; second, if this is not the case, then whether it would be possible to modify our tableaux so that, on the one hand, the number of rules and labels remained constant across all ETL- and NFL-like logics based upon $\mathbf{Mn}$ lattices, and on the other hand, they would provide for improvement on truth tables and the tableaux we have presented here.

These questions, however, are outside of the scope of our paper, and we leave them for future research.

On the other hand, one of the ways to present tableaux for both finite- and infinite-valued logics is to use sets of truth values (‘constraints’) as labels (cf.~e.g.,~\cite{Haehnle1992,Haehnle1994,Haehnle1999}). Note that labels $\sim$ and $\nsim$ for pairs of formulas act as a kind of constraints. However, one of the drawbacks of constraint tableaux is that the number of possible constraints grows with the number of truth values. In our approach, we have only two constraints for any \textbf{Mn} lattice.

Finally, we note that the existence of finite analytic tableaux for $\mathbf{M}\omega$ should imply that the complexity of deciding $\vdash_{\mathrm{ETL}_{\mathbf{M}\omega}}$ and $\vdash_{\mathrm{NFL}_{\mathbf{M}\omega}}$ is in co-$\mathcal{NP}$ since we can non-deterministically guess and check a valuation falsifying $\phi\vdash_{\mathrm{ETL}_{\mathbf{M}\omega}}\chi$ ($\phi\vdash_{\mathrm{NFL}_{\mathbf{M}\omega}}\chi$, respectively) in polynomial time. It is an open question, though, whether $\mathrm{ETL}_{\mathbf{M}\omega}$ and $\mathrm{NFL}_{\mathbf{M}\omega}$ can be characterised by (some other) finite matrix. Indeed, as it is shown in~\cite[Theorem~3.17.]{CaleiroMarcelinoRivieccio2018} a logic is characterised by a single finite matrix iff it is characterised by some matrix, finitely determined, and there can be only finitely many pairwise non-equivalent under $\dashv\vdash$ formulas over a finite set of variables. We leave this question for further research.
\section{Conclusion and future work}\label{conclusion}
In this paper we have presented ETL- and NFL-like logics on \textbf{M3} lattices and considered analytic tableaux systems for them. We have also generalised our tableaux for the family of ETL- and NFL-like logics on \textbf{Mn} lattices without adding new rules or new kinds of labelled formulas.
	
First and foremost, non-distributive ETL and NFL are still relatives of first degree entailment on the one hand and on the other hand they are based upon lattices of truth values. Hence, it would be interesting to introduce modal operators for these logics as done, e.g. in~\cite{CastiglioniErtola-Biraben2017} and~\cite{Hartonas2015}.
	
Second, there are known modal extensions for FDE considered in~\cite{Priest2008,Priest2008FromIftoIs}, \cite{OmoriWansing2017}, and~\cite{OdintsovWansing2017}. It would be interesting to study modal extensions for ETL- and NFL-like logics.
	
Third, it could be fruitful to study connection of non-distributive ETL and NFL logics with modal logics as it has been recently done in~\cite{Kubyshkina2019} for first degree entailment.
	
Fourth, we have used extensively that all ‘middle-level’ values are virtually indistinguishable in the sense that they all are either designated or undesignated. Hence, it would be interesting to study non-distributive relatives of FDE where entailment relations would coincide with ‘upward’ ordering on the \textbf{Mn} lattices. This way we will need to somehow circumvent the fact that ‘middle-level’ values in those non-distributive FDE-like logics would be distinguishable by their entailment relations.

Fifth, we have shown that ETL- and NFL-like logics on $\mathbf{M}\omega$ are indeed the logics of all bounded lattices for which conditions listed in Definition~\ref{ETLNFLlike} apply. One of them is that negation of $x\in L$ can produce its complement only if $x\in\{\top,\bot\}$. This condition is present in all of \textbf{K3}, LP, ETL, NFL and FDE, so it seemed natural to include it into the list of properties of ETL- or NFL-like properties. However, one might ask, whether $\mathrm{ETL}_{\mathbf{M}\omega}$ and $\mathrm{NFL}_{\mathbf{M}\omega}$ will still be ETL- and NFL-like logics over all bounded lattices if we get rid of this condition.

Sixth, as we have noted in the previous section, our tableaux feature some superficial similarities with systems using analytic cut rule. It is reasonable to ask whether these similarities are indeed only superficial or do our tableaux provide improvement over truth table as \textbf{KE} systems do. If our tableaux are not better than truth tables, it would be interesting to construct a \textbf{KE} system which could be generalised onto a wide family of ETL- and NFL-like logics without adding new rules or labels.

Finally, although, it was relatively easy to provide an easy generalisable tableaux calculus for a wide family of ETL- and NFL-like logics, we have not presented any ‘formula-formula’ axiomatisations. The question, hence, could be raised, how to construct an elegant family of such calculi which differ from one another only in some certain axioms.
\bibliographystyle{plain}
\bibliography{references.bib}
\end{document}